\documentclass[a4paper, 10pt]{amsart}
\usepackage{graphicx, color} 
\usepackage{amsmath, amsthm, amsfonts, amssymb, amscd}
\usepackage[T1]{fontenc}

\theoremstyle{plain}
\newtheorem{Teo}{Theorem}[section]
\newtheorem{Def}[Teo]{Definition}

\newtheorem{Lema}[Teo]{Lemma}
\newtheorem{Prop}[Teo]{Proposition}
\newtheorem{Cor}[Teo]{Corollary}

\newtheorem{maintheorem}{Theorem}

\begin{document}

\title[Statistical properties for equilibrium states]{On statistical properties for equilibrium states of partially hyperbolic horseshoes}

\author{V. Ramos }
\address{Vanessa Ramos \\ Centro de ci\^{e}ncias exatas e tecnologia-ufma\\ Av. dos Portugueses, 1966, Bacanga\\  65080-805 São Lu\'{i}s\\Brasil}
\email{vramos@impa.br}

\author{J. Siqueira}
\address{Jaqueline Siqueira\\ Centro de Matem\'{a}tica da Universidade do Porto\\ Rua do
Campo Alegre 687\\ 4169-007 Porto\\ Portugal}
\email{jaqueline.rocha@fc.up.pt}

\date{}
\thanks{The authors were supported by CNPq-Brazil.}

\pagenumbering{arabic}

\begin{abstract}
We derive some statistical properties for equilibrium states of partially hyperbolic horseshoes. We define a projection map associated to the horseshoe and prove a spectral gap for its transfer operator  acting on the space of H\"older continuous observables. From this we deduce an exponential decay of correlations and a central limit theorem. We finally extend these results to the horseshoe. 

\end{abstract}

\maketitle

\tableofcontents




\section{Introduction}
\label{introducao}


Describing the behavior of the orbits of a dynamical system can be a challenging task, especially for systems that have a complicated topological and geometrical structure. A very useful way to obtain features of such systems is via invariant probability measures. For instance, by Birkhoff's Ergodic Theorem, almost every initial condition in each ergodic component of an invariant measure  has the same statistical distribution in space. 
When the system admits more than one invariant probability measure, an efficient way to chose an interesting one is to select those that have regular Jacobians, which are called equilibrium states. We formally define an equilibrium state with respect to a potential as follows.

\begin{Def}\label{est eq}
Consider a continuous map $F: \Omega \to \Omega$ on a compact metric space $\Omega$. We say 
that an $F$-invariant probability measure $\mu$ is an equilibrium state for $F$ w.r.t. a 
continuous potential $\phi: \Omega \to \mathbb{R}$ if it satisfies
          $$h_{\mu} (F) + \int \phi \, d\mu  = \sup_\eta \left\{ h_{\eta} (F) + \int \phi \, d\eta 
\right\},  $$
where the supremum is taken over all $F$-invariant probability measures.
\end{Def}

By studying the decay of correlations of an equilibrium measure, one can obtain significant information regarding  the system: how fast memory of the past is lost by the system as time evolves. In particular, this gives the speed at which the equilibrium is reached.

However, while standard counterexamples show that in general there is no specific rate at which
this loss of memory occurs, it is sometimes possible to obtain specific rates of decay which depend only on the map $F$, as long as the observables belong to some appropriate space of functions.

Another way to characterize weak correlations of successive observations is  given by a central limit theorem: the probability of a given deviation of the average values of an observable along an orbit from the asymptotic average is essentially given by a normal distribution.


In a pioneering work \cite{FS}, Ferrero and Schmitt applied the theory of projective metrics, due to Birkhoff \cite{Birkhoff}, to the transfer operator for expanding maps, thus obtaining spectral properties. For one dimensional piecewise expanding maps, an exponential decay of correlations was proved by Liverani \cite{Liverani2} and a central limit theorem was proved by Keller \cite{Ke}.   
In the context of volume preserving hyperbolic maps, Liverani \cite{Liverani} established  exponential decay of correlations for the SRB measure. In the more general context of hyperbolic attractors, Viana \cite{Viana} proved the exponential decay of correlations and a central limit theorem. The latter was inspired by the work of D\"urr and  Goldstein \cite{DG}.

In the context of non-uniformly hyperbolic maps we may cite the independent works of Young \cite{Young} and Keller and Nowicki \cite{KN} that used towers extensions and cocycles to prove exponential decay of correlations for quadratic maps. In the same context Castro and Varandas \cite{CV} obtained statistical properties for the unique equilibrium state associated to a class of non-uniformly expanding maps. In this work they use the projective metrics approach. 

For a class of partially hyperbolic systems semiconjugated to nonuniformly expanding maps Castro and Nascimento \cite{CN} proved exponential decay of correlations and a central limit theorem for the maximal entropy measure. 

In this work we address the problem of studying statistical properties for the unique equilibrium state of partially hyperbolic horseshoes. The family of three dimensional horseshoes was introduced by D\'iaz, Horita, Rios and Sambarino in \cite{diazetal} and the uniqueness of equilibrium states associated to H\"older continuous potentials with small variation was proved by Rios and Siqueira in \cite{RS15}. 

We start by studying a two dimensional abstract map obtained from the horseshoe by projecting  its inverse on two center-stable leaves. We refer to this map as the {\em projection map}. We construct  metrics  with respect to which the Perron-Frobenius operator associated to the projection map is a contraction. Such a contraction allows us to obtain a spectral gap property on the space of H\"older continuous observables.  From this we deduce exponential decay of correlations and a central limit theorem for the equilibrium state associated to the projection map.
Finally we show that the equilibrium state of the horseshoe carries the same statistical properties.

The paper is organized as follows.
In Section~\ref{Definitions} we describe both  the horseshoe and its projection map and we give a precise formulation of the statistical properties of its equilibrium. We also define the transfer operator associated to the projection map and state the spectral gap property. In Section~\ref{cones} we give a brief review of the theory of projective metrics in cones. This will be used as a key  tool to obtain the spectral gap theorem, which we prove in Section~\ref{gap}. In Section~\ref{decay for G} we derive the exponential decay of correlations and a central limit theorem for the unique equilibrium of the projection map. Finally, in Section~\ref{decay for F} we extend the results obtained for the projection map to the horseshoe.  


\section{Definitions and main results}\label{Definitions}

We start  describing the family of three dimensional horseshoes introduced by D\'iaz, Horita, Rios and Sambarino in  \cite{diazetal}.   
Let $R= [0,1]\times[0,1]\times[0,1]\subset\mathbb{R}^3$ be the cube in $\mathbb{R}^3$ and consider the parallelepipeds: 
$$
\tilde{R}_0 =[0,1]\times [0,1]\times [0,1/6]
\qquad \mbox{and} \qquad
\tilde{R} _1=[0,1]\times [0,1]\times [5/6,1].
$$ 
The horseshoe map is defined on $\tilde{R}_0$ and $\tilde{R}_1$ as follows  
   $$ F_{0}(x,y,z):=F_{| \tilde{R}_{0} }(x,y,z) =(\rho x , f(y),\beta z),$$
where $0 < \rho <{1/3}$, $\beta> 6$ and  $f(y) =\frac {1}{1 - (1-\frac{1}{y})e^{-1}}. $ 
                      
\begin{figure}[!htb]
\centering
\includegraphics[scale=0.3]{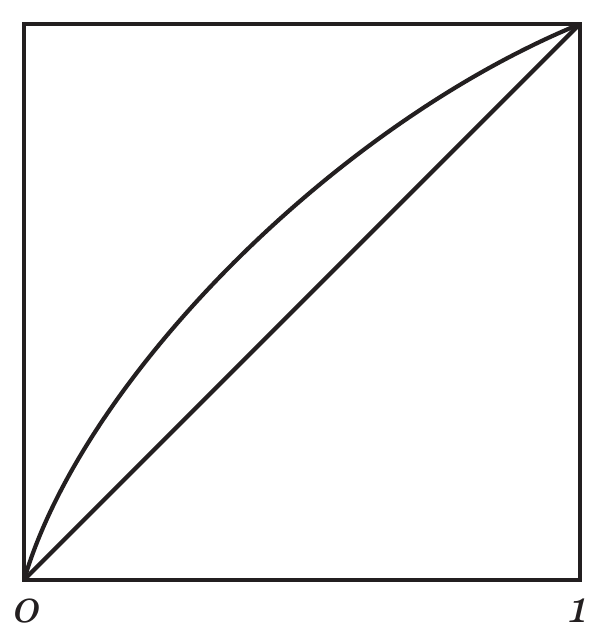}
\caption{The function $f$.}\label{tempoum}
\end{figure}
And 
$$F_{1}(x,y,z)=F_{| \tilde{R}_{1} }(x,y,z) = \Big(\frac{3}{4}- \rho x , \sigma (1 - y) ,\beta_{1} \Big(z - \frac{5}{6} \Big)\Big),$$
where  $0<\sigma< {1/3}$ and $3< \beta_1 < 4$.

Then, for $X \in R$, we have 
\begin{equation}
\label{Fdeles}
F(X) = \left\{
\begin{array}{lcl}
F_0(X) & \mbox{if} & X \in \tilde R_0\\
F_1(X) & \mbox{if} & X \in \tilde R_1.
\end{array}
\right.
\end{equation}

If $X \in R$ but does not belong to $\tilde{ R}_0 $ or $\tilde{ R}_1 $, then $X$ will be mapped injectively outside $R$.

We point out that besides we refer simply to $F$, we have described a family of maps that depends on the parameters $\rho, \beta, \beta_1$ and $\sigma$. We consider fixed parameters satisfying conditions above.  

In figure~\ref{ferradura2} we see the steps of the construction of the horseshoe.  

\begin{figure}[!htb]
\centering
\includegraphics[scale=0.4]{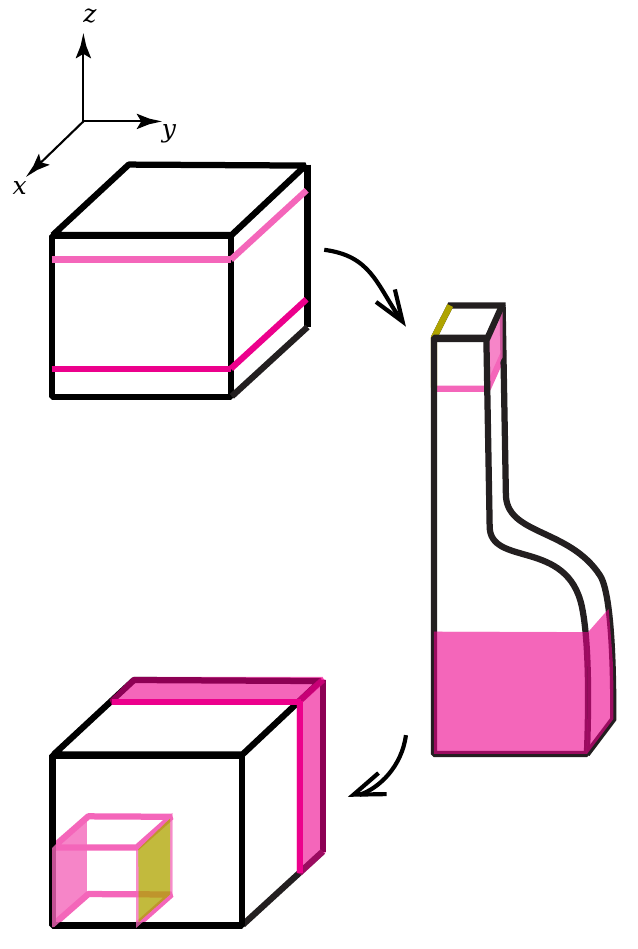}
\caption{The horseshoe $F$}\label{ferradura2}
\end{figure}
Let $\Omega$ be the maximal invariant set under $F$ of the union of the parallelepipeds $\tilde{ R}_0$ and $ \tilde {R}_1$:
$$\Omega = \bigcap_{n\in \mathbb{Z}} F^{n}(\tilde{ R}_0 \cup \tilde {R}_1).$$

In \cite{diazetal} it was shown that  the maximal invariant set $\Omega$ is partially hyperbolic, with one dimensional central direction, parallel to the $y$-axis. The central direction presents contractive and expanded  behavior. The horizontal direction is contractive while the vertical direction, parallel to the $z$-axis, is expanding.

The uniqueness of equilibrium states for the horseshoe $F$ associated to potentials with small variation was proved in \cite{RS15}. The main goal of this work is to study the statistical behavior of this equilibrium. Here we state the result in \cite{RS15}. Let $\omega= \frac{1 + \sqrt{5}}{2}$.  

\begin{Teo}
Let $F: \tilde{R}_0 \cup \tilde{R}_1 \to R$ be the three dimensional partially hyperbolic horseshoe defined above. Let $ \phi: \tilde{R}_0 \cup \tilde{R}_1  \to \mathbb{R}$ be a H\"older continuous potential with $\sup \phi - \inf \phi < \frac{\log{\omega}}{2}$. Assume that $\phi$ does not depend on the $z$-coordinate in each set $\tilde{R}_0$ and $\tilde{R}_1 $. Then there exists a unique equilibrium state $\mu_{\phi}$ for the system $F$ with respect to the potential $\phi$.
\end{Teo}

We consider potentials $\phi$ as above and additionally we assume that the H\"older constant of $\exp(\phi)$ is small. The explicit condition will be stated in Section~\ref{gap}. We point out that this is an open condition which includes, for instance, constant potentials.   

For the equilibrium state $\mu _{\phi}$ of the system $(F,\phi)$  we will establish exponential decay of correlations for H\"older continuous observables. 

\begin{maintheorem}\label{decaimento F}
The equilibrium state $\mu_{\phi}$ has exponential decay of correlations for H\"older continuous observables: there exists  a constant $0<\tau < 1$ such that for all $\varphi \in L^{1}(\mu_{\phi}), \psi \in C^{\alpha}(\tilde{R}_0\cup\tilde{R}_1)$ there exists $K:=K(\varphi, \psi)>0$ satisfying 
$$\left|  \int \left(\varphi \circ F^n\right)\psi \, d\mu_{\phi} - \int \varphi \, d\mu_{\phi}\int\psi \, d\mu_{\phi} \right|   \leq \ K \cdot \tau^{n}  \quad \mbox{for every} \ n\geq 1. $$
\end{maintheorem}

We also derive a central limit theorem for the equilibrium state of the horseshoe with respect to a potential $\phi$ as considered above.  

\begin{maintheorem}\label{TCL F}
Let $\varphi $ be a H\"older continuous function and let
 $\sigma \geq 0$ be defined by 
    $$\sigma^{2}= \int \psi^2  \ d\mu_{\phi} + 2\displaystyle\sum_{n=1}^{\infty}  \int \psi (\psi \circ F^n) \ d\mu_{\phi} \quad \mbox{where} \quad \psi = \varphi - \int \varphi \ d\mu_{\phi} .$$
    Then $\sigma$ is finite and $\sigma = 0$ if and only if $\varphi = u \circ F - u $ for some $u \in L^{2}(\mu_\phi)$. On the other hand, if $\sigma >0 $ then given any interval $A\subset \mathbb{R}$,
 $$\mu_{\phi}\left\{  x\in \tilde{R}_0 \cup \tilde{R}_1 : \frac{1}{\sqrt{n}} \displaystyle\sum_{j=0}^{n} \left(  \varphi(F^j(x))   - \int \varphi \ d\mu_{\phi} \right) \in A \right\} \to \frac{1}{\sigma \sqrt{2\pi}} \int_{A} e^{-\frac{t^2}{2\sigma^2}} \ dt,$$ 
 as $n$ goes to infinity.
\end{maintheorem}

Now we describe a map $G$ that was defined in \cite{RS15} which is related to the projection of $F^{-1}$ on two center-stable planes. By an abuse of notation the map $G$ will be called the \emph{projection map}. Besides the inherent interest in the dynamics of the map $G$,  understanding the statistical behavior of its equilibrium is the crucial ingredient in the proofs of  Theorem~\ref{decaimento F} and Theorem~\ref{TCL F}.

We define as follows the rectangles $R_1$, $R_2$ and $R_3$:
\[ 
\begin{split}
& R_1 = [0,\rho]\times [0,1]\times \{0\}, \\
& R_2 = [3/4 - \rho,3/4]\times [0,\sigma]\times \{0\}, \\
& R_3 = [0,\rho]\times [1+ \varepsilon,2+ \varepsilon]\times \{5/6 \},
\end{split}
\]
with $\varepsilon >0$ close to zero.
\begin{figure}[!htb]
\centering
\includegraphics[scale=0.4]{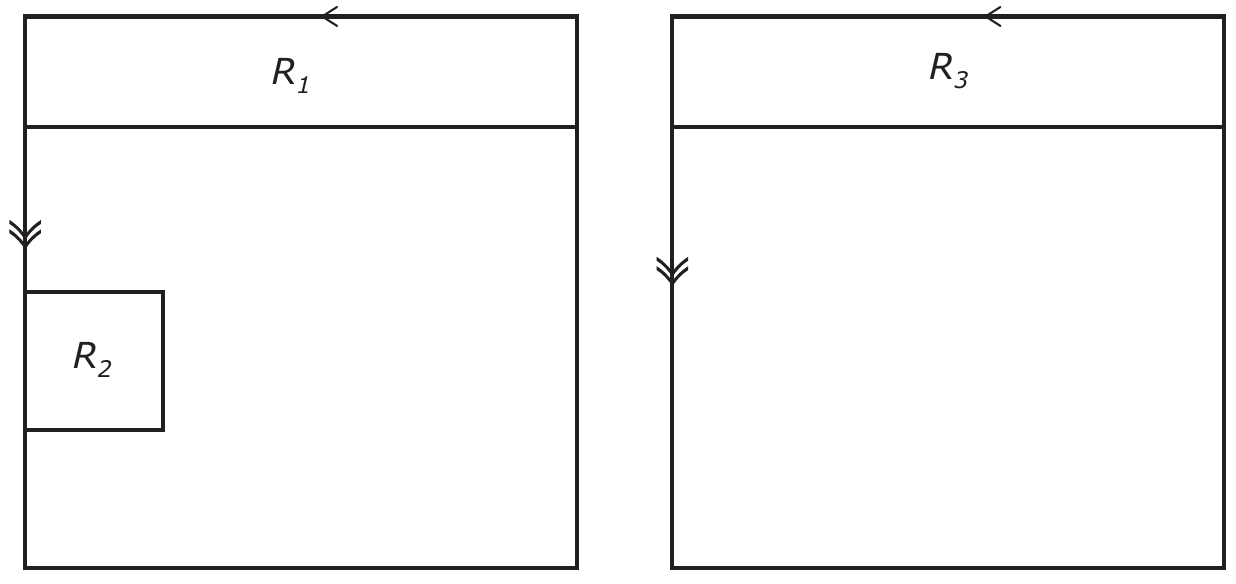}
\caption{The rectangles $R_1$, $R_2$ and $R_3$.}\label{retangulos}
\end{figure}

The rectangles are inside two planes that we will call $P_0$ and $P_1$ (see figure~\ref{retangulos}). We consider an abstract space  $\mathcal{Q} := \bigcup_{i=1}^{3}{R_{i}} $ which is the union of the rectangles. Notice this is a metric space endowed with some natural metric $d$, say the one, induced by $\mathbb{R}^3$.

Let  $g_0, g_1: [0,1] \to \mathbb{R}$ be defined by $g_0(y)=f^{-1}(y) $ and  $g_1(y)= 1- \sigma^{-1}y$. Take $\gamma = \rho^{-1}$.  

Consider the map $G:\mathcal{Q} \to  P_{0} \cup  P_{1}$ defined by its restrictions $G_i$ to each rectangle $R_i$ as follows:
\begin{eqnarray*}
  G_{1}(x,y,z) & = & \big(\gamma x, g_{0}(y),0  \big), \\        
  G_{2}(x,y,z) & = & \big(\gamma (3/4 - x),g_{1}(y) ,5/6  \big), \\ 
  G_{3}(x,y,z) & = & \big(\gamma x,g_{0}(y),0  \big). 
\end{eqnarray*}   

Note that $R_2$ is uniform expanding while we have both, expanding and contracting, behaviors in $R_1$ and $R_3$. 
The map $G$ acts similarly  on $R_1$ and $R_3$. In these rectangles, the points are sent from the right side to the left, except for the extreme points whose $x$ coordinates are fixed. 

Let $\Lambda$ be the maximal invariant set under $G$ of the union of the rectangles $R_1$, $R_2$ and $R_3$: 
\begin{equation*}
\label{maxinv}
\Lambda := \displaystyle\bigcap _{n \in \mathbb{N}}G^{-n} (\mathcal{Q} ),
\end{equation*}
and from now on we denote simply by $G$ the restriction of $G$ to $\Lambda$.
 
Let $\Sigma _{A}$  be the subshift of finite type  
$$\Sigma_{A}= \left\{ \Theta =  \left( \theta _{0} \theta _{1} \theta _{2} \cdots \right) \in \left\{1,2,3 \right\} ^ {\mathbb{N}} | A_{\theta_{i}\,\theta_{i+1}} = 1 \right\}, $$
with transition matrix:
$$ A = \left( 
\begin{array}{ccc}
   1 & 1 & 0 \\
   0 & 0 & 1 \\
   1 & 1 & 0 \\      
\end{array}
\right).  $$

The following transitions are allowed:
\begin{eqnarray*}
   & &   1 \rightarrow 1, 2 \\
   & &   2 \rightarrow 3\\
   & &   3 \rightarrow 1, 2 .
\end{eqnarray*} 

Notice that $G$ is is not conjugated but only semi-conjugated to the shift $\sigma$. That is because the entire segment $[0,1]$ is associated to the constant sequence $(11111 \cdots)$ on $\Sigma_{A}$.

We point out that the 3-rd iterate of any rectangle covers $\mathcal{Q}$.
Moreover, $G$ is topologically mixing.

Note that points belonging to the rectangles 1 and 2 have two pre-images, while points in the rectangle 3 have just one pre-image.

\begin{figure}[htb] \centering 
\begin{minipage}[b]{0.6\linewidth}
\includegraphics[width=\linewidth]{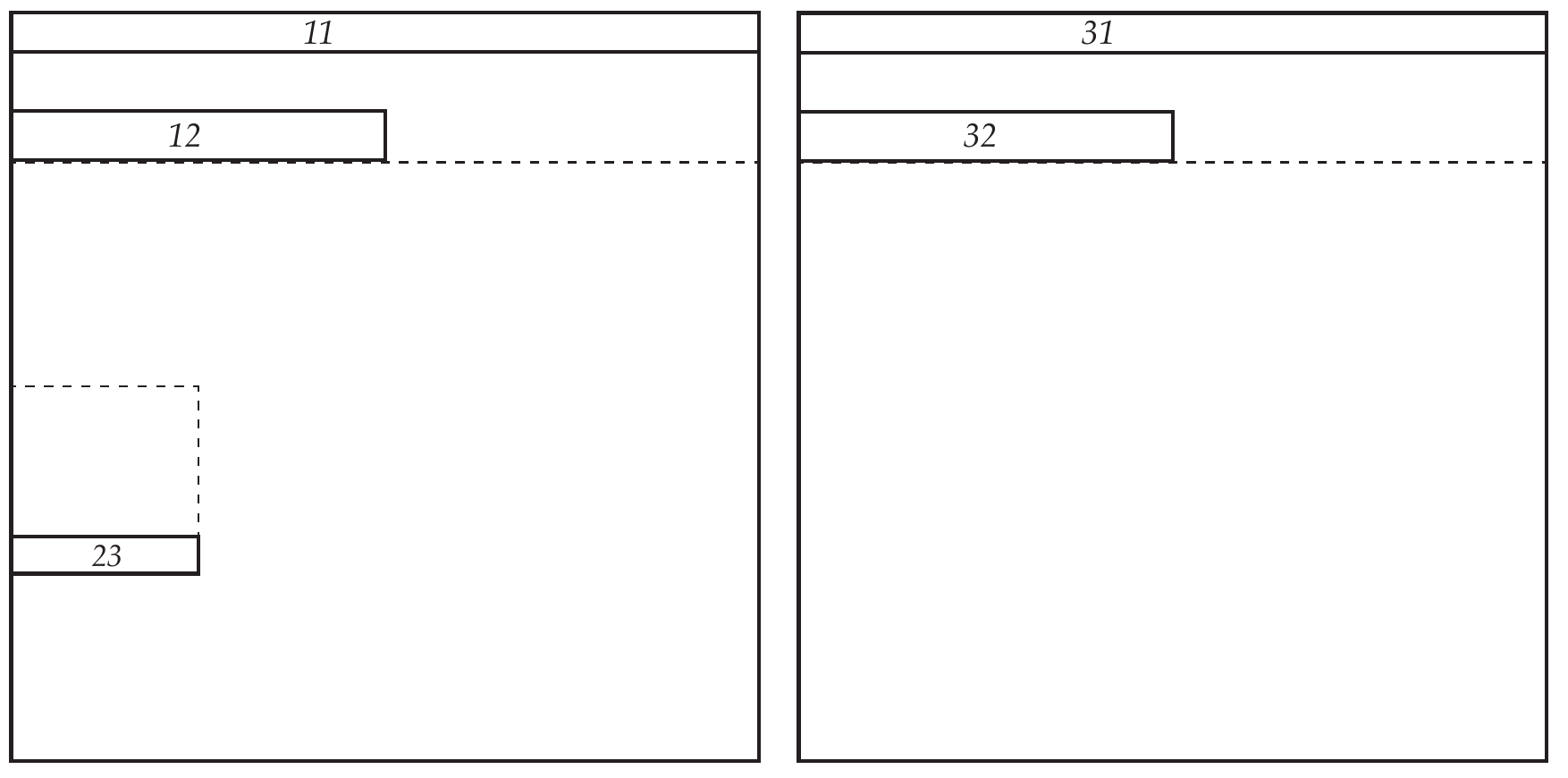} \caption{Second generation}\label{2geracao}
\end{minipage} 
\hfill
\begin{minipage}[b]{0.6\linewidth} 
\includegraphics[width=\linewidth]{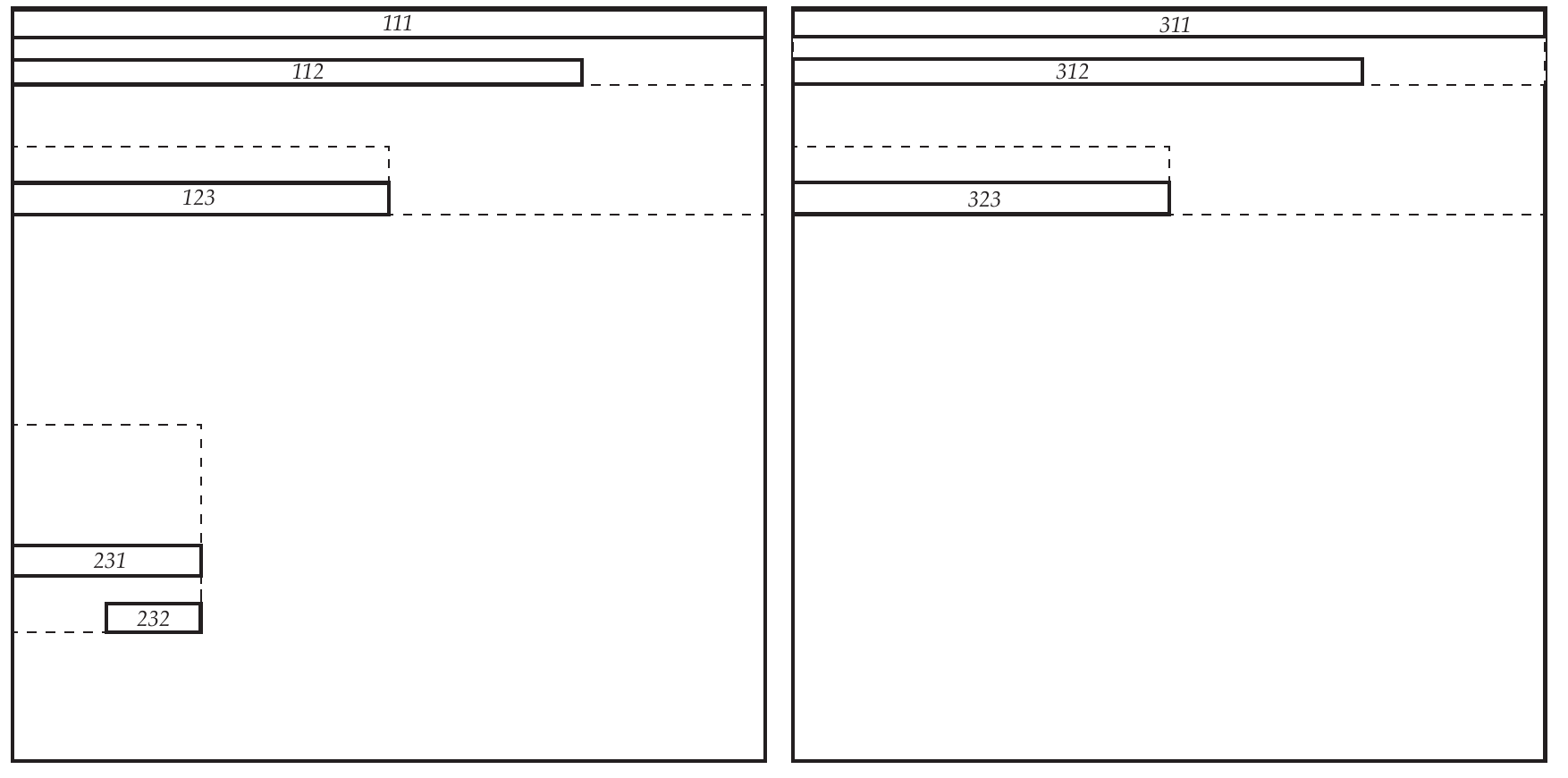} \caption{Third generation}\label{3geracao} 
\end{minipage}
\end{figure}

Figure~\ref{2geracao} and figure~\ref{3geracao} show the  first steps in the generation of the set $\Lambda$ . Since $\Lambda$ contains infinitely many line segments it is not a Cantor set.

The topological entropy of the subshift $\sigma$ is given by:
 $$h_{top}(\sigma)=\log \left(\frac{1+\sqrt{5}} {2}\right)=\log \omega.$$ 

Since $G$ and $\sigma$ are  semiconjugated we obtain $h_{top}(G) \geq \log \omega$.

In \cite{RS15} it was shown the uniqueness of equilibrium states associated to H\"older continuous potentials $\phi_{\ast}: \mathcal{Q} \to \mathbb{R}$  satisfying
$\sup \phi_{\ast} - \inf \phi_{\ast} < \frac{\log \omega}{2}$. In this work we consider potentials $\phi_{\ast}$ as above and assume an additional condition that will be stated in Section~\ref{gap}. For the system $(G, \phi_{\ast})$ we obtain some statistical properties of its equilibrium measure $\mu_{\ast}$. As mentioned before, these results will be used to derive the statistical properties of the equilibrium of the horseshoe announced above.

The following result states the exponential decay of correlations for H\"older continuous observables.

\begin{maintheorem}\label{decaimento G} 
There exists  a constant $0<\tau < 1$ such that for all $\varphi \in L^{1}(\mu_{\ast})$ and $ \psi \in C^{\alpha}(\mathcal{Q})$ there exists $K:=K(\varphi, \psi)>0$ satisfying 
$$\left| \int \left(\varphi \circ G^n\right)\psi \, d\mu_{\ast} - \int \varphi \, d\mu_{\ast}\int \psi \, d\mu_{\ast} \right| \leq \ K \cdot \tau^{n}  \quad \mbox{for every} \ n\geq 1. $$
\end{maintheorem}

We also obtain a central limit theorem for the equilibrium.

\begin{maintheorem}\label{TCL G}
Let $\varphi $ be a H\"older continuous function and let  $\sigma \geq 0$ be defined by 
    $$\sigma^{2}= \int \psi^2  \ d\mu_{\ast} + 2\displaystyle\sum_{n=1}^{\infty}  \int \psi (\psi \circ G^n) \ d\mu_{\ast} \quad \mbox{where} \quad \psi = \varphi - \int \varphi \ d\mu_{\ast} .$$
    Then $\sigma$ is finite and $\sigma = 0$ if and only if $\varphi = u \circ G - u $ for some $u \in L^{2}(\mu_{\ast})$. On the other hand, if $\sigma >0 $ then given any interval $A\subset \mathbb{R}$,
 $$\mu_{\ast}\left\{ x\in \mathcal{Q} : \frac{1}{\sqrt{n}} \displaystyle\sum_{j=0}^{n -1} \left(  \varphi(G^j(x))   - \int \varphi \ d\mu_{\ast} \right) \in A\right\} \to \frac{1}{\sigma \sqrt{2\pi}} \int_{A} e^{-\frac{t^2}{2\sigma^2}} \ dt,$$ 
 as $n$ goes to infinity.
\end{maintheorem}


\subsection{Ruelle-Perron-Frobenius operator and its spectral gap}
\label{operador transf}

Let $(G, \phi_{\ast})$ be the system defined  above. Denote by $C^{0}(\mathcal{Q})$ the set of real continuous functions on $\mathcal{Q}$. We define the operator $\mathcal{L}_{\phi_{\ast}} : C^{0}\left( \mathcal{Q} \right) \rightarrow C^{0}\left(\mathcal{Q} \right)$  called the {\emph{Ruelle-Perron-Frobenius operator}} or simply the {\emph{transfer operator}}, which associates to each $\psi \in C^{0}(\mathcal{Q})$ a continuous function
$ \mathcal{L}_{\phi_{\ast}} (\psi) \colon \mathcal{Q} \to \mathbb{R}$ by:
$$\label{optransf}
\mathcal{L}_{\phi_{\ast}} \psi \left(x\right) = \displaystyle\sum _{y  \in \, G^{-1}\left(x\right)} e^{\phi_{\ast}(y)} \psi \left( y \right). 
$$

The transfer operator $\mathcal{L}_{\phi_{\ast}}$ is a positive bounded linear operator. 
For each $n \! \in\! \mathbb{N}$ we have  
$$\label{iteradostransf}
\mathcal{L}_{\phi_{\ast}}^{n} \psi \left(x\right) = \displaystyle\sum _{y  \in \, G^{-n}\left(x\right)} e^{S_{n}\phi_{\ast} \left(y \right)} \psi \left( y \right),
$$
where $S_{n}\phi_{\ast}$ denotes the Birkhoff sum $S_{n}\phi_{\ast}(x)= \displaystyle\sum_{j=0}^{n-1} \phi_{\ast}\big(G^{j}(x)\big)$.

We also consider the dual operator $ \mathcal{L}_{\phi_{\ast}}^{\ast}: \mathcal{M}(\mathcal{Q}) \to \mathcal{M}(\mathcal{Q})$ that satisfies
$$\int \psi \ d\mathcal{L}_{\phi_{\ast}}^{\ast}\eta  = \int  \mathcal{L}_{\phi_{\ast}}(  \psi ) \ d\eta , $$
for every $\psi \in  C^{0}(\mathcal{Q}) $ and every $\eta \in \mathcal{M}(\mathcal{Q})  $.

We will state here for further reference an important property for the transfer operator and its dual which was obtained in \cite{RS15}.

\begin{Teo} Let $\lambda$ be the spectral radius of the transfer operator $\mathcal{L}_{\phi_{\ast}}$. There exist a probability measure $\nu\in\mathcal{M}(\mathcal{Q})$ and a H\"older continuous function $h:\mathcal{Q}\rightarrow\mathbb{R}$ bounded away from zero and infinity which satisfies
$$\mathcal{L}_{\phi_{\ast}}^{\ast}\nu=\lambda\nu \quad and \quad \mathcal{L}_{\phi_{\ast}}h=\lambda h.$$
\end{Teo}

We point out that the unique equilibrium state $\mu_{\ast}$ associated to the system $(G,\phi_{\ast} )$ is given by $\mu_{\ast}= h\nu$.

The Ruelle-Perron-Frobenius operator $\mathcal{L}_{\phi_{\ast}}$ is said to have the \emph{spectral gap property} if its spectrum $\sigma(\mathcal{L}_{\phi_{\ast}})\subset \mathbb{C} $ can be decomposed as follows: $\sigma(\mathcal{L}_{\phi_{\ast}}) =\left\{ \lambda_0 \right\} \cup \Sigma_0$ where $\lambda_0 \in \mathbb{R}$ is an eigenvalue for $\mathcal{L}_{\phi_{\ast}}$ associated to a one-dimensional eigenspace and $ \Sigma_0$ is strictly contained in the ball $\left\{ z\in \mathbb{C}: |z|< \lambda_0 \right\} $.

\begin{maintheorem}\label{gap spectral}
The Ruelle-Perron-Frobenius operator $\mathcal{L}_{\phi_{\ast}}$ has the spectral gap property restrict to the space of H\"older continuous observables.  
\end{maintheorem}


\section{Invariant cones and projective metrics} \label{cones}
The theory of projective metrics on convex cones and positive operators on a vector space is due to Birkhoff \cite{Birkhoff} and has been extensively studied (see \cite{Baladi} and \cite{Liverani}). Projective metrics associated to cones  provide an elegant way to express spectral properties of the transfer operator.

In this section we will state some results regarding this theory in order to prove the spectral gap of the transfer operator.

Let $E$ be a Banach space. A subset $\mathcal{C}$ of $E  \!-\! \{0\}$ is called a \emph{cone} in $E$ if it is a convex space which satisfies:
\begin{enumerate}
\item $\forall \lambda >  0 : \lambda \mathcal{C} \subset \mathcal{C}; $
\item $ \mathcal{C} \cap  \left( - \mathcal{C}\right) = \{\emptyset \}.$
\end{enumerate}
We say that a cone $\mathcal{C} $ is closed if $\bar{\mathcal{C}}= \mathcal{C} \cup \{0\}$.

Let $\mathcal{C}$ be a closed cone and given $v, w \in \mathcal{C} $ define 
\begin{equation}\label{A e B}
A(v,w)= \sup \left\{t>0 : w-tv \in \mathcal{C}  \right\} \ \mbox{and} \ B(v,w)=\inf \left\{s>0 : sv -w \in \mathcal{C}  \right\}.
\end{equation}

We point out that $A(v,w)$ is finite, $B(v,w)$ is positive and $A(v,w)\leq B(v,w)$ for all $v, w \in \mathcal{C}$. 
We set  

$$ \Theta(v,w) = \log\left( \frac{B(v,w)}{A(v,w)} \right) $$ 
    
with $\Theta$ possibly infinity in the case $A=0$ or $B=+\infty$.    

It is straightforward to check that $\Theta(v,w)$ is well-defined and takes values in $[0,+\infty]$. Since $\Theta(v,w)= 0 \Leftrightarrow v= tw \ \mbox{for some} \ t>0$ we have that $\Theta$ defines a pseudo-metric on $\mathcal{C}$. Then $\Theta$ induces a metric on a projective quotient space of $\mathcal{C}$ called  the \emph{projective metric of} $\mathcal{C}$.  

Note that the projective metric depends in a monotone way on the cone: if $\mathcal{C}_1 \subset \mathcal{C}_2$ are two cones in $E$, then we have
     $$\Theta_2(v,w) \leq \Theta_1(v,w) \quad \mbox{for all} \quad v,w \in \mathcal{C}_1   $$
where $\Theta_1$ and $\Theta_2$ are the projective metrics in $\mathcal{C}_1$ and $\mathcal{C}_2$  respectively.  

Moreover, if $L:{E}_1 \to {E}_2 $ is a linear operator and $\mathcal{C}_1, \mathcal{C}_2$ are cones in ${E}_{1}, {E}_{2}$ respectively, satisfying $L(\mathcal{C}_{1}) \subset \mathcal{C}_{2}$ then $$\Theta_2(L(v),L(w)) \leq \Theta_1(v,w) \quad \mbox{for all} \quad v,w \in \mathcal{C}_1. $$

However $L$ is not necessarily a strict contraction, that will be the case for instance if  $L(\mathcal{C}_{1})$ had finite diameter in $\mathcal{C}_{2}$. This will be stated in the following result which is a key tool to prove the spectral gap for the Ruelle-Perron-Frobenius operator. 

\begin{Prop}\label{cont viana}
Let $\mathcal{C}_{1}$ and $\mathcal{C}_{2}$ be closed convex cones in the Banach spaces ${E}_1$ and ${E}_2$ respectively. If $L:E_1 \to E_2 $ is a linear operator satisfying $L(\mathcal{C}_{1}) \subset \mathcal{C}_{2}$ and  $\Delta = {\rm diam}_{\Theta_2}(L(\mathcal{C}_{1})) >0$ then
$$\Theta_2 \left(L(\varphi), L(\psi) \right) \leq (1-e^{-\Delta}) \cdot \Theta_1 \left( \varphi, \psi \right) \quad \mbox{for all} \ \varphi, \psi \in \mathcal{C}_{1}.$$
\end{Prop}

For the proof of the last proposition see for example [\cite{Viana}, Proposition 2.3].

Next we will define a cone in the space of positive continuous functions. We start by recalling some definitions.

Let $\varphi$ be an $\alpha$-H\"older continuous function and denote by 
$$|\varphi|_{\alpha}=\sup_{x\neq y}\frac{|\varphi(x)-\varphi(y)|}{d(x, y)^\alpha},$$
the H\"older constant of $\varphi$.

Given $\delta>0$ we say that a function $\varphi$ is  $(C, \alpha)$-H\"older continuous in balls of radius $\delta $ if for some constant $C>0$  we have $|\varphi(x)-\varphi(y)|\leq Cd(x, y)^\alpha$ for all $y \in B(x, \delta)$. 

We will denote by $|\varphi|_{\alpha,\delta}$ the smallest H\"older constant of $\varphi$ in balls of radius $\delta>0$.
We consider the space of $\alpha$-H\"older continuous observables endowed with the norm $\| \cdot \| := | \cdot|_0 + |\cdot|_{\alpha}$.

Consider $\mathcal{Q}$ the union of the rectangles $R_1$, $R_2$ and $R_3$ and fix $1/2 \leq \delta \leq 3/4 - 2\rho$. Let $\varphi: \mathcal{Q} \to \mathbb{R}$ be a 
 $(C, \alpha)$-H\"older continuous function in balls of radius $\delta$ .  Then $\varphi$ is $(C(1+r^{\alpha}), \alpha)$-H\"older continuous in balls of radius $(1+r)\delta$ for each $0\leq r \leq 1.$

Indeed, fixing $r\in[0, 1]$ and given $x, y\in \mathcal{Q}$ with $d(x, y)<(1+r)\delta$, there exists $z\in \mathcal{Q}$ such that $d(x, z)=\delta$ and $d(z, y)<rd(x, z)$. Hence,
\begin{eqnarray}\label{eq do r}
\left |\varphi(x)-\varphi(y)\right |&\leq& \left |\varphi(x)-\varphi(z)\right |+\left |\varphi(z)-\varphi(y)\right | \nonumber  \\
&\leq & Cd(x, z)^{\alpha}+Cd(z, y)^{\alpha}\leq C(1+r^{\alpha})d(x, y)^{\alpha}.
\end{eqnarray}

The next result states that every locally H\"older continuous function defined on $\mathcal{Q}$ is H\"older continuous.
\begin{Lema}\label{bolas} Let $\delta>1/2$ and let
 $\varphi: \mathcal{Q} \to \mathbb{R}$ be a $(C, \alpha)$-H\"older continuous function in balls of radius $\delta$. Then there exists $m=m(\delta)>0$ such that $\varphi$ is $(m\cdot C, \alpha)$-H\"older continuous.
\end{Lema}
\begin{proof} By the compactness of $\mathcal{Q}$, there exists $N\in\mathbb{N}$ which depends only on $\delta$ such that given $x, y\in \mathcal{Q}$ there are $z_{0}=x, z_{1},..., z_{N+1}=y$ with $d(z_{i}, z_{i+1})\leq\delta$ for all $i=0,\cdots, N$ and $d(z_{i}, z_{i+1})\leq d(x, y).$

Since $\varphi$ is $(C, \alpha)$-H\"older continuous in balls of radius $\delta$ it follows that
$$\left |\varphi(x)-\varphi(y)\right |\leq \sum_{i=0}^{N} \left | \varphi(z_{i})-\varphi(z_{i+1}) \right |\leq \sum_{i=0}^{N} C d(z_{i}, z_{i+1})^{\alpha}\leq C(N\!+\!1)d(x, y)^{\alpha}.$$

Thus $\varphi$ is $(m\cdot C, \alpha)$-H\"older continuous where $m=N\!+\!1$.
\end{proof}

Now we consider the cone of locally H\"older continuous observables defined on $\mathcal{Q}$: 
\begin{equation} \label{cone holder}
\mathcal{C}_{k,\delta}= \left\{ \varphi : \varphi>0 \ \mbox{and} \  \frac{|\varphi|_{\alpha,\delta}}{\inf \varphi} \leq k \right\}.
\end{equation}

It follows by definition that $\mathcal{C}_{k_1,\delta} \subset \mathcal{C}_{k_2,\delta}$ if $k_1 \leq k_2$.

Given an arbitrary $\varphi \in \mathcal{C}_{k,\delta} $ we have $\left|\varphi \right|_{\alpha,\delta} \leq  k \cdot \inf \varphi$. Moreover, by Lemma~\ref{bolas}, $\varphi$ is H\"older continuous with constant $m \cdot \left|\varphi \right|_{\alpha,\delta} $. Then 
\begin{equation}  \label{sup no cone}
\sup{\varphi} \leq \inf{\varphi}  + m \left|\varphi \right|_{\alpha,\delta} \cdot \left[\mbox{diam} (\mathcal{Q} )\right]^{\alpha} \leq \left[ 1+ m  \cdot k \cdot  \left[\mbox{diam} (\mathcal{Q} )\right]^{\alpha} \right] \inf{\varphi}.
\end{equation}

In the next lemma we give another expression for the projective metric on the cone  $\mathcal{C}_{k,\delta}$, that we denote by $\Theta_k$ and use in further estimates.

\begin{Lema}\label{metrica cone}
The metric $\Theta_k$ in the cone $\mathcal{C}_{k,\delta}$ is given by $\Theta_k(\varphi , \psi)= \log \left( \frac{B_k(\varphi, \psi)}{ A_k(\varphi, \psi) }\right)  $ where  
$$A_k(\varphi, \psi):=\displaystyle\inf_{d(x,y)< \delta, z\in \mathcal{Q}} \frac{k|x-y|^{\alpha}\psi(z) - (\psi(x)-\psi(y))}{k|x-y|^{\alpha}\varphi(z) - (\varphi(x) - \varphi(y))} $$
 and $$B_k(\varphi, \psi) := \displaystyle\sup_{d(x,y)< \delta, z\in \mathcal{Q}} \frac{k|x-y|^{\alpha}\psi(z) - (\psi(x)-\psi(y))}{k|x-y|^{\alpha}\varphi(z) - (\varphi(x) - \varphi(y))}.$$
\end{Lema}
\begin{proof}
First recall the definition of the projective metric and consider $A$ and $B$ as in equation~\eqref{A e B}. Let $\varphi,\psi \in \mathcal{C}_{k,\delta} $. Let $A(\varphi, \psi)= A$ be the supremum of positive numbers satisfying  
  $\psi - A\varphi \in \mathcal{C}_{k,\delta}  $.
This is equivalent to saying that $\psi(x) - A\varphi(x) >0 $ for all $x\in \mathcal{Q}$ and $|\psi - A\varphi|_{\alpha,\delta} \leq k \inf(\psi -A\varphi)$. Hence
\begin{equation}\label{eq A}
A(\varphi, \psi) \leq \min \left\{ \inf_ {x \in \mathcal{Q}} \frac{\psi(x)}{\varphi(x)} , \displaystyle\inf_{0<d(x,y)<\delta , z\in \mathcal{Q}} \frac{k|x-y|^{\alpha}\psi(z)- (\psi(x)-\psi(y)}{k|x-y|^{\alpha}\phi(z)- (\phi(x)-\phi(y)}  \right\}.
\end{equation}
Suppose that the minimum can be attained by the first term on the right side of the inequality. In this case, we can take $x_0$ satisfying 
   $$  \inf_ {x \in \mathcal{Q}} \frac{\psi(x)}{\varphi(x)}=   \frac{\psi(x_0)}{\varphi(x_0)} .$$ 
Thus, for every $x\in \mathcal{Q}$ we have 
\begin{eqnarray*}
&&\psi(x_0)[\phi(x)- \phi(x_0)]- \phi(x_0)[\psi(x)- \psi(x_0)] \leq  0 \\
&\Rightarrow & \!\! \phi(x_0)[k(x-x_0)^{\alpha}\psi(x_0)- \! \! (\psi(x)\! -\psi(x_0))] \leq \\
&&\phi(x_0)[k(x-x_0)^{\alpha}\phi(x_0)- (\phi(x) -\phi(x_0))] \\
&\Rightarrow & \!\! \frac{k|x-x_0|^{\alpha}\psi(z)- (\psi(x)-\psi(x_0)}{k|x-x_0|^{\alpha}\phi(z)- (\phi(x)-\phi(x_0)} \leq \frac{\psi(x_0)}{\varphi(x_0)}.
\end{eqnarray*} 
This guarantees that the minimum in equation~\eqref{eq A} is always attained by the right-hand side. To end the proof just notice that a similar computation can be done to get the expression for $B$.
\end{proof}


\section{Spectral gap}\label{gap}
Consider the system $(G,\phi_{\ast})$ where $G: \mathcal{Q} \to P_0 \cup P_1$ is the projection map defined in Section~\ref{Definitions} and $\phi_{\ast}: \mathcal{Q} \to \mathbb{R}$ is a H\"older continuous potential with variation smaller than  $\frac{\log(\omega)}{2}$. We also assume that $\phi_{\ast}$ satisfies 
\begin{equation}\label{condição potencial}
e^{3var\phi_{\ast}}e^{2\alpha} \left(\frac{2}{3} e^{2\alpha} + \sigma^{\alpha} \right) + \frac{10 e^{3\alpha}}{3} m (\mbox{\mbox{diam}}\mathcal{Q})^\alpha \frac{| e^{3\phi_{\ast}}|_{\alpha} }{e^{3\inf \phi_{\ast}}}<1
\end{equation}

Let $\mathcal{L}_{\phi_{\ast}}$ be the transfer operator of $G$ associated to the potential $\phi_{\ast}$. When there is no risk of confusion we will denote $\mathcal{L}_{\phi_{\ast}}$ simply by $\mathcal{L}$. 
Here we prove that the transfer operator $\mathcal{L}$ has the spectral gap property on the space of H\"older continuous observables.

As mentioned in Section~\ref{Definitions}, the map $G$ is not injective on $\mathcal{Q}$ but it is injective when restricted to each of the rectangles $R_1$, $R_2$ and $R_3$. Moreover, $G^{3}(R_i) \supset \mathcal{Q}$ for each $i=1,2,3$. We will explore this property in order to construct a partition of $\mathcal{Q}$ such that the distance between pre-images under $G^{3}$ of points in the same element of the partition can be controlled.  

\begin{Lema}\label{delta}
There exists a finite cover $\mathcal{P}$ of $\mathcal{Q}$ by  injective domains of  $G^{3}$ such that every $x\in \mathcal{Q}$ has at most one pre-image in each element of the cover $\mathcal{P}$. Moreover, for $\frac{1}{2} \leq \delta \leq \frac{3}{4} - 2\rho $  we have that $d(x,y)< \delta$  implies 
                 $$d(x_i,y_i) < (1+r)\delta   \quad \mbox{for some}  \ r<1$$
where $x_i, y_i$ are pre-images of $x, y$ under $G^3$ belonging to the same element of $\mathcal{P}$.				\end{Lema}
\begin{proof}
Let $\mathcal{Q}= \bigcup_{i=1}^{3} R_i$ be the union of the rectangles and consider the partition  $\mathcal{P}= \bigvee_{j=0}^{3}G^{-j}(\mathcal{Q})$. Thus $\mathcal{P}$ is a finite cover  of $\mathcal{Q}$ with $13$ elements and satisfies that $G^3$ is injective in each $P_i \in \mathcal{P}$, $i=1, \cdots, 13$. 

Let $\frac{1}{2} \leq \delta \leq \frac{3}{4} - 2\rho $. Given $x,y \in \mathcal{Q}$ with $d(x,y) < \delta$ we have, in particular, that $x, y$ belong to the same rectangle $R_j$, $j=j_{x,y} \in \{1, 2,3\}$. 

Since  $\left\| DG^{-1}|_{R_1}   \right\| = \left\| DG^{-1}|_{R_3}   \right\| = e $  and $\left\| DG^{-1}|_{R_2}   \right\| = \sigma < 1$ there is no loss of generality in assuming that $x, y \in R_1$. Let $G^{-3}(x) = \left\{x_i| i=1, \cdots, 5 \right\}$ and $G^{-3}(y) = \left\{y_i| i=1, \cdots, 5 \right\}$. Let $x_1, y_1$ be the pre-images in the element $P_1 := \bigcap_{j=0}^3 G^{-j}(R_1)$. We have
   $d(x_i, y_i) \leq d(x_1, y_1)  \quad \mbox{for all} \ i=1, \cdots, 5. $
   
Since $G^{3}$ preserves horizontal and vertical lines and $DG^{-1}|_{R_1}$ contracts vertical lines we may assume that $x, y$ are on the same horizontal line.		
Again, using that the derivative in $P_1$ can be estimated by $e$, it is straightforward to check that  $d(x_1, y_1) < \frac{e^3}{\delta (e^3 -1)+1} \delta.$
Since $\delta \geq \frac{1}{2}$ we have $r:= \left(\frac{e^3}{\delta (e^3 -1)+1}-1\right) < 1 $.
\end{proof}

From now on we fix $\delta >0$ as in Lemma~\ref{delta}. Given $k>0$ let $\mathcal{C}_{k,\delta}$ be the cone of locally H\"older continuous functions defined in equation~\ref{cone holder}. The next result states the strict invariance of the cone $\mathcal{C}_{k,\delta}$ under the operator $\mathcal{L}_{\phi_{\ast}}^{3}$ for $k$ large enough.

\begin{Prop}\label{invariancia}
There exists $0<\hat{\lambda}<1$ such that  $$ \mathcal{L}_{\phi_{\ast}}^{3}(\mathcal{C}_{k,\delta}) \subset \mathcal{C}_{\hat{\lambda}k,\delta} \qquad \mbox{for} \ k>0 \ \mbox{sufficiently large}.$$
\end{Prop}
\begin{proof}
Let $k>0$ and take $\varphi\in \mathcal{C}_{k,\delta}$ with constant $C=C(\varphi):= |\varphi|_{\alpha, \delta}$. Since $\mathcal{L}_{\phi_{\ast}}^{3}$ is a positive and bounded operator  we have that $\mathcal{L}_{\phi_{\ast}}^{3}(\varphi)$ is a continuous and positive function. In order to prove the result we should find $0<\hat{\lambda}<1$ such that  
$$\frac{ |\mathcal{L}_{\phi_{\ast}}^{3}(\varphi)|_{\alpha, \delta}}{\inf ( \mathcal{L}_{\phi_{\ast}}^{3}(\varphi) )} \leq \hat{\lambda} \cdot k.$$
 
Notice that given $x\in \mathcal{Q}$ we have $3 \leq \# \left\{G^{-3}(x)\right\} \leq 5$ and thus 
\begin{eqnarray*}
\mathcal{L}_{\phi_{\ast}}^{3}(\varphi)(x)= \sum_{y\in G^{-3}(x)} e^{S_3\phi_{\ast}(y)} \cdot \varphi(y) \geq 3\cdot e^{3\inf \phi_{\ast}} \cdot \inf \varphi. 
\end{eqnarray*}

Considering $x, y \in \mathcal{Q} $ with $d(x,y)< \delta$ we have that $x,y$ belong to the same rectangle and, in particular, they have the same number of pre-images. So we can group the pre-images that are in the same rectangle. Thus
\begin{eqnarray*}
&&\frac{ \left|  \mathcal{L}_{\phi_{\ast}}^{3}(\varphi)\right|_{\alpha, \delta}}{\inf \left( \mathcal{L}_{\phi_{\ast}}^{3}(\varphi) \right)} \leq  \frac{\left| \mathcal{L}_{\phi_{\ast}}^{3}(\varphi)(x) - \mathcal{L}_{\phi_{\ast}}^{3}(\varphi)(y) \right|} {\inf \mathcal{L}_{\phi_{\ast}}^{3}\varphi \cdot d(x,y)^{\alpha} } \\ \\
&\leq &  \frac{  \displaystyle\sum_{i=1}^{5} | e^{S_3\phi_{\ast}(x_i)} \varphi(x_i)  -  e^{S_3\phi_{\ast}(y_i)}  \varphi(y_i)  |} { \inf \mathcal{L}_{\phi_{\ast}}^{3}\varphi \cdot d(x,y)^{\alpha}} \\ \\
&\leq&  \frac{ \displaystyle\sum_{i=1}^{5} \left| e^{S_3\phi_{\ast}(x_i)}\right| \left| \varphi(x_i) - \varphi(y_i)  \right|} { \inf \mathcal{L}_{\phi_{\ast}}^{3}\varphi \cdot d(x,y)^{\alpha}}  + \frac{\displaystyle\sum_{i=1}^{5} \left|\varphi(y_i)\right|  \left| e^{S_3\phi_{\ast}(x_i)} - e^{S_3\phi_{\ast}(y_i)}  \right|}{\inf \mathcal{L}_{\phi_{\ast}}^{3}\varphi \cdot d(x,y)^{\alpha}} \\ \\
&\leq&   \frac{   e^{3\sup\phi_{\ast}} (2 e^{3\alpha} (1+r^{\alpha})   +  3  e^{2\alpha}  \sigma^{\alpha})\cdot C    d(x,y)^{\alpha}  }{3\cdot e^{3\inf \phi_{\ast}} \inf \varphi \cdot  d(x,y)^{\alpha}} 
 +  \frac{ 5 \sup\varphi \cdot | e^{3\phi_{\ast}}|_{\alpha} \cdot e^{3\alpha}  d(x,y)^{\alpha}}{3\cdot e^{3\inf \phi_{\ast}} \cdot \inf \varphi \cdot d(x,y)^{\alpha}} \\ \\
&\leq & e^{3 var\phi_{\ast}} \left[ \frac{2}{3} e^{3\alpha} (1+r^{\alpha}) + e^{2\alpha} \cdot \sigma^{\alpha}  \right]  k  + \frac{5e^{3\alpha}| e^{3\phi_{\ast}}|_{\alpha} }{3e^{3\inf \phi_{\ast}}} \left[ 1+ mk (\mbox{\mbox{diam}}\mathcal{Q})^{\alpha}\right]    \\
&\leq& \left[ e^{3var\phi_{\ast}}e^{2\alpha} \left(\frac{2}{3} e^{2\alpha} +   \sigma^{\alpha} \right) + \frac{10 e^{3\alpha}}{3} m (\mbox{diam}\mathcal{Q})^\alpha \frac{| e^{3\phi_{\ast}}|_{\alpha} }{e^{3\inf \phi_{\ast}}} \right] k.
\end{eqnarray*}

Observe that in the fourth inequality we used equation~\eqref{eq do r} and in the fifth inequality we applied equation~\eqref{sup no cone}. 
By condition~\eqref{condição potencial} we obtain that the last inequality is smaller than $\hat{\lambda}k$ for some positive constant $\hat{\lambda} < 1$.
\end{proof}

The invariance of the cone is not enough to guarantee that  the operator $\mathcal{L}^{3}_{\phi_{\ast}}$ is a contraction. In order to prove this we have to verify that the cone $\mathcal{C}_{\hat{\lambda}k, \delta}$ given by the previous proposition has finite diameter.

\begin{Prop}\label{diam finito}
The cone $\mathcal{C}_{\hat{\lambda}k,\delta}$ has finite diameter for $k>0$ sufficiently large.
\end{Prop}

\begin{proof}

Given an arbitrary $\varphi \in \mathcal{C}_{\hat{\lambda}k,\delta} $ we have $\left|\varphi \right|_{\alpha,\delta} \leq \hat{\lambda}  \cdot k \cdot \inf \varphi$. By equation~\eqref{sup no cone}: 
\begin{equation}\label{sup no cone2}  
\sup{\varphi} \leq \inf{\varphi}  + m \left|\varphi \right|_{\alpha,\delta} \cdot \left[\mbox{diam} (\mathcal{Q} )\right]^{\alpha} \leq \left[ 1+ m \cdot \hat{\lambda} \cdot k \cdot  \left[\mbox{diam} (\mathcal{Q} )\right]^{\alpha} \right] \inf{\varphi}.
\end{equation}
Given $\varphi, \psi \in \mathcal{C}_{\hat{\lambda}k,\delta} $ by Lemma~\ref{metrica cone} one can obtain the following estimate
\begin{eqnarray*}
\Theta_k(\varphi, \psi) \leq \log \left( \frac{k \cdot \sup\varphi + \hat{\lambda} \cdot k \cdot \inf\varphi}{k \cdot \inf\varphi - \hat{\lambda} \cdot k \cdot \inf\varphi } \cdot \frac{k \cdot \sup\psi + \hat{\lambda} \cdot k \cdot \inf\psi}{k \cdot \inf\psi - \hat{\lambda} \cdot k \cdot \inf\psi}  \right).
\end{eqnarray*}
Using equation~\eqref{sup no cone2} we have
\begin{eqnarray*}
\Theta_k(\varphi, \psi) &\leq& \log \left( \frac{k(1+ m \cdot \hat{\lambda} \cdot k \left[\mbox{diam} (\mathcal{Q} )\right]^{\alpha})(1+ \hat{\lambda} ) \inf \varphi}{k(1-\hat{\lambda}) \inf\varphi}  \right)  \\
&+& \log \left( \frac{k(1+ m \cdot \hat{\lambda} \cdot k \left[\mbox{diam} (\mathcal{Q} )\right]^{\alpha})(1+ \hat{\lambda} ) \inf \psi}{k(1-\hat{\lambda}) \inf\psi}  \right) \\
   &\leq&  2\log\left( \frac{1+ \hat{\lambda} }{1- \hat{\lambda} }  \right) +  2\log\left( 1+ m\cdot \hat{\lambda} \cdot k  \left[\mbox{diam} (\mathcal{Q} )\right]^{\alpha} \right).
\end{eqnarray*}
Since $\varphi$ and $\psi$ are arbitrary it implies that the diameter of  $\mathcal{C}_{\hat{\lambda} k, \delta}$ is finite.
\end{proof}

Combining Proposition~\ref{invariancia} and Proposition~\ref{diam finito} we are able to apply Proposition~\ref{cont viana} to establish the next result.

\begin{Prop} \label{contracao no cone}
The operator $\mathcal{L}^3$ is a contraction in the cone $\mathcal{C}_{k,\delta}$: for the constant $\Delta = \mbox{diam}(\mathcal{C}_{\hat{\lambda}k,\delta}) >0$ we have
   $$\Theta_{k} \left(\mathcal{L}_{\phi_{\ast}}^{3}(\varphi), \mathcal{L}_{\phi_{\ast}}^{3}(\psi) \right) \leq (1-e^{-\Delta}) \cdot \Theta_k \left( \varphi, \psi \right) \quad \mbox{for all} \ \varphi, \psi \in \mathcal{C}_{k, \delta}.$$	
\end{Prop}

\vspace{0.3cm}

As in Subsection~\ref{operador transf} we consider the function $h$ and the measure $\nu$ satisfying $\mathcal{L}_{\phi_{\ast}} h= \lambda h$ and $\mathcal{L}_{\phi_{\ast}}^{\ast} \nu = \lambda \nu$. Also recall that $\mu_{\ast}= h\nu$. From the last proposition we will derive  exponential convergence of the transfer operator to the eigenfunction $h$ in the space of H\"older continuous observables.

\begin{Prop} \label{cota norma}
For every $\varphi \in \mathcal{C}_{k, \delta}$ satisfying $\int \varphi \ d\nu = 1$ there exist some positive constant $L$ and $0<\tau <1$ such that $$\left\| \lambda^{-n}\mathcal{L}^n_{\phi_{\ast}}(\varphi) - h \right\| = \left|  \lambda^{-n}\mathcal{L}^n_{\phi_{\ast}}(\varphi) - h \right|_{0}  + \left|   \lambda^{-n}\mathcal{L}^n_{\phi_{\ast}}(\varphi) - h\right| _{\alpha, \delta} \leq L\tau^n   \quad \forall n\geq 1.$$ 
\end{Prop}
\begin{proof}
Let $\varphi \in \mathcal{C}_{k, \delta}$ with $\int \varphi \ d\nu = 1$. Since $\nu$ is the reference measure associated to $\lambda$ and $\mu_\ast=h\nu$ we have for every $j \geq 1$
    $$ \int \lambda^{-j}\mathcal{L}^j_{\phi_{\ast}}(\varphi) \ d\nu = \int \lambda^{-j}\varphi \ d(\mathcal{L}^j_{\phi_{\ast}})^{\ast}\nu = \int \varphi \ d\nu =1=\int h\ d\nu .$$
Thus, for every $j\geq 1$ we derive 
$$\inf\frac{\lambda^{-j}\mathcal{L}^j_{\phi_{\ast}}(\varphi)}{\lambda^{-j}\mathcal{L}^j_{\phi_\ast}(h)}=\inf \frac{\lambda^{-j}\mathcal{L}^j_{\phi_{\ast}}(\varphi)}{h} \leq 1\leq \sup \frac{\lambda^{-j}\mathcal{L}^j_{\phi_{\ast}}(\varphi)}{h}=\sup\frac{\lambda^{-j}\mathcal{L}^j_{\phi_{\ast}}(\varphi)}{\lambda^{-j}\mathcal{L}^j_{\phi_\ast}(h)}.$$		

Let $\tilde{\mathcal{L}}_{\phi_{\ast}}=\lambda^{-3}\mathcal{L}^{3}_{\phi_{\ast}}$ and $\tau= 1-e^{-\Delta}$ where
$\Delta=\mbox{diam}(\mathcal{C}_{\hat{\lambda}k,\delta})$. From Lemma~\ref{metrica cone} and Proposition~\ref{contracao no cone} we have
\begin{eqnarray} \label{cota sup}
\nonumber e^{-\Delta \tau^{j}}\!\!\leq A_{k}(\tilde{\mathcal{L}}^j_{\phi_{\ast}}(\varphi), \tilde{\mathcal{L}}^j_{\phi_{\ast}}(h))\!\!\!&\leq&\!\!\!\inf \frac{\tilde{\mathcal{L}}^j_{\phi_{\ast}}(\varphi)}{h}\\
\nonumber\!\!\!&\leq&\!\!\! 1\\
\!\!\!&\leq&\!\!\! \sup \frac{\tilde{\mathcal{L}}^j_{\phi_{\ast}}(\varphi)}{h} \leq B_{k}(\tilde{\mathcal{L}}^j_{\phi_{\ast}}(\varphi), \tilde{\mathcal{L}}^j_{\phi_{\ast}}(h)) \leq e^{\Delta \tau^{j}}.
\end{eqnarray}
Thus for all $j\geq 1$, we have:
$$\left|  \tilde{\mathcal{L}}^j_{\phi_{\ast}}(\varphi) - h \right|_{0} \leq \left| h \right|_{0} \left| \frac{\tilde{\mathcal{L}}^j_{\phi_{\ast}}(\varphi)}{h} -1 \right|_{0} \leq \left| h \right|_{0} \left( e^{\Delta \tau^j} -1 \right) \leq L_{1}\tau ^j . $$
Moreover, the inequality \eqref{cota sup} also gives us 
$$e^{-\Delta\tau^{j}}\leq A_k(\tilde{\mathcal{L}}^j_{\phi_{\ast}}(\varphi), \tilde{\mathcal{L}}^j_{\phi_{\ast}}(h)) \leq \frac {kd(x,y)^{\alpha} \tilde{\mathcal{L}}^j_{\phi_{\ast}}(\varphi)(z)  - \left(\tilde{\mathcal{L}}^j_{\phi_{\ast}}(\varphi)(x) - \tilde{\mathcal{L}}^j_{\phi_{\ast}}(\varphi)(y)  \right)}{k d(x,y)^{\alpha}h(z) - \left(  h(x) -h(y)\right)},$$\\
and
$$\frac {kd(x,y)^{\alpha} \tilde{\mathcal{L}}^j_{\phi_{\ast}}(\varphi)(z)  - \left(\tilde{\mathcal{L}}^j_{\phi_{\ast}}(\varphi)(x) - \tilde{\mathcal{L}}^j_{\phi_{\ast}}(\varphi)(y)  \right)}{k d(x,y)^{\alpha}h(z) - \left(  h(x) -h(y)\right)}\leq B_k (\tilde{\mathcal{L}}^j_{\phi_{\ast}}(\varphi), \tilde{\mathcal{L}}^j_{\phi_{\ast}} (h)) \leq e^{\Delta\tau^{j}}.$$\\
Therefore for every $j \geq 1$ we obtain
\begin{eqnarray*}
&&\!\!\!\!\left|  \tilde{\mathcal{L}}^j_{\phi_{\ast}}(\varphi) - h\right| _{\alpha} = \sup_{x\neq y } \frac{ \left( \tilde{\mathcal{L}}^j_{\phi_{\ast}}(\varphi)(y) -h(y)\right) - \left( \tilde{\mathcal{L}}^j_{\phi_{\ast}}(\varphi)(x) -h(x)\right)}{d(x,y)^{\alpha}} \\
&\leq &\!\!\!\!\left| \tilde{\mathcal{L}}^j_{\phi_{\ast}}(\varphi) - h\right| _{0} + \\
&+&\!\!\!\! \left|{ \frac{kd(x,y)^{\alpha} \tilde{\mathcal{L}}^j_{\phi_{\ast}}(\varphi)(z) - \left( \tilde{\mathcal{L}}^j_{\phi_{\ast}}(\varphi)(x) - \tilde{\mathcal{L}}^j_{\phi_{\ast}}(\varphi)(y) \right)}{kd(x,y)^{\alpha}h(z) - \left( h(y) - h(x)  \right) }  - 1}\right|\!\!\cdot\!\!\left|{ kh(z) - \frac{h(y) - h(x)}{d(x, y)^{\alpha}}} \right|     \\\\
&\leq&\!\!\!\!  L_{1}\tau^j + ( e^{\Delta \tau^j} - 1) \cdot \left( k\left|h\right|_{0}+\left| h \right|_{\alpha}\right) \leq L_{2}\tau^{j}. 
\end{eqnarray*}

Thus for every $j\geq 1$ we have the inequality
$$\left\| \tilde{\mathcal{L}}^{j}_{\phi_{\ast}}(\varphi) - h \right\| = \left| \tilde{\mathcal{L}}^{j}_{\phi_{\ast}}(\varphi) - h \right|_{0}  + \left|  \tilde{\mathcal{L}}^{j}_{\phi_{\ast}}(\varphi) - h\right| _{\alpha, \delta} \leq L_{3}\tau^j.$$ 

Now, given  $n \geq 1$ write $n= 3j +r$ with $j<n $ and $0 \leq r <3$. Since $\mathcal{L}_{\phi_{\ast}}$ is a bounded operator and $\mathcal{L}_{\phi_{\ast}}h= \lambda h$, we conclude that
\begin{eqnarray*}
\left\| \lambda^{-n}\mathcal{L}_{\phi_{\ast}}^{n}(\varphi)  -h  \right\| &=& \left\| \lambda^{-r}\mathcal{L}_{\phi_{\ast}}^{r}\left( \lambda^{-3j}\mathcal{L}_{\phi_{\ast}}^{3j}  - h \right)   \right\| \\
&\leq&  \left\| \lambda^{-1}\mathcal{L}_{\phi_{\ast}} \right\| ^{r} \cdot \left\| \tilde{\mathcal{L}}_{\phi_{\ast}}^{j}(\varphi) -h   \right\|   \leq L\tau^n .
\end{eqnarray*}
\end{proof}

As a consequence of the exponential convergence we can prove the following  property of the equilibrium state associated to the system $(G, \phi_{\ast})$.

\begin{Cor}
The sequence $\left(G^n_{\ast}\nu\right)_{n \in \mathbb{N}}$ of push forwards of the reference measure converges to the equilibrium state $\mu_{\ast}$.
\end{Cor}
\begin{proof}
Let $\varphi \in C^{0}(\mathcal{Q})$ be arbitrary. Since $\mathcal{L}^{\ast}\nu= \lambda \nu$ we have
\begin{eqnarray*}
\int \varphi \ d\, G^n_{\ast} \nu &= & \int \varphi \circ G^{n}   \ d\nu = \lambda^{-n} \int \varphi  \circ G^{n}  \ d\left(\mathcal{L}^{\ast} \right) ^{n}\nu \\
&=& \lambda^{-n} \int \mathcal{L}^{n} \left( \varphi  \circ G^{n} \right)   d\nu = \lambda^{-n} \int \varphi  \mathcal{L}^{n} \left( \mathbf{1}  \right)   d\nu.
\end{eqnarray*}
By Proposition~\ref{cota norma} the sequence $ \left\{\lambda^{-n}\mathcal{L}^{n} \left(  \mathbf{1}  \right) \right\}$ converges uniformly to $h$ which implies
$$\displaystyle\lim_{n \to \infty} \int \varphi \ d\, G^n_{\ast} \nu = \int \varphi h  d\nu = \int \varphi  d\mu_{\ast}  $$
and ends the proof.
\end{proof}

To finish this section we prove the spectral gap of the transfer operator.  

\begin{Teo}
The operator $\mathcal{L}_{\phi_{\ast}}$ acting on the space $C^{\alpha}\left(\mathcal{Q}\right)$ admits a decomposition of its spectrum: there exists $0 < r_0 <\lambda$ such that $\Sigma = \left\{ \lambda \right\} \cup \Sigma_0 $ with $\Sigma_0$ contained in a ball $B(0, r_0)$ centered at zero and of radius $r_0$.
\end{Teo}
\begin{proof}
Consider $\tilde{\mathcal{L}}_{\phi_{\ast}} = {\lambda}^{-1}\mathcal{L}_{\phi_{\ast}}$.
Define $E_0 = \left\{ \psi \in C^{\alpha}\left(\mathcal{Q} \right) : \int \psi \ d\nu = 0 \right\}$ and $E_1$ as the eigenspace associated to the eigenvalue $1$. Notice that ${\rm dim} \, E_1=1$. 

We can decompose $C^{\alpha}\left(\mathcal{Q}\right)$ as a direct sum of $E_0$ and $E_1$. In fact, given $\varphi \in C^{\alpha}\left(\mathcal{Q} \right)$  write 
$$\varphi = \left[ \varphi - \int\varphi \ d \nu \cdot h \right] + \left[ \left( \int \varphi \ d\nu \right) h \right]= \varphi_0 + \varphi_1$$ 
and notice that $\varphi_0:= \left[ \varphi - \int\varphi \ d \nu  \right]$ belongs to $E_0$ and $\varphi_1 = \left[ \left( \int \varphi \ d\nu \right) h \right]$ belongs to $E_1$ since $\int h\ d\nu=1$.
Thus in order to derive the spectral gap for $\tilde{\mathcal{L}}_{\phi_{\ast}}$ in $C^{\alpha}\left(\mathcal{Q}\right)$ it is enough to prove that $\tilde{\mathcal{L}}_{\phi_{\ast}}^n$ is a contraction in $E_0$ for $n$ sufficiently large. 

We take $k>0$ large enough such that the cone $\mathcal{C}_{k,\delta }$ is preserved by $\tilde{\mathcal{L}}_{\phi_{\ast}}$. 
Take $\varphi \in E_0$ with $\left|\varphi \right|_{\alpha, \delta} \leq 1$. So $\varphi$ does not necessarily belong to the cone  but $\left(\varphi +2\right) \in \mathcal{C}_{k, \delta}$ since $$\frac{\left|\varphi + 2 \right|_{\alpha. \delta}}{\inf\left(\varphi +2 \right)}  = \frac{\left|\varphi \right|_{\alpha. \delta}}{\inf\left(\varphi +2 \right)} \leq  \frac{1}{\inf\left(\varphi +2 \right)} \leq k \quad \mbox{for} \  k  \ \mbox{large}  . $$
Thus by Proposition~\ref{cota norma} we have 
\begin{eqnarray*} 
\left\| \tilde{\mathcal{L}}_{\phi_{\ast}}^n(\varphi) \right\| &=& \left\| \tilde{\mathcal{L}}_{\phi_{\ast}}^n(\varphi + 2) - \tilde{\mathcal{L}}_{\phi_{\ast}}^n(2)\right\|\\ \\
&\leq& \left\| \tilde{\mathcal{L}}_{\phi_{\ast}}^n(\varphi + 2) - 2h \right\| + \left\| \tilde{\mathcal{L}}_{\phi_{\ast}}^n(2) - 2h \right\| \\\\
&\leq& \left\| \left( \int \varphi + 2 \ d\nu \right) \mathcal{L}_{\phi_{\ast}}^n\left(\frac{\varphi + 2}{\int \varphi + 2 \ d\nu}\right) -2h \right\| + \left\| \tilde{\mathcal{L}}_{\phi_{\ast}}^n(2) - 2h \right\| \\\\
&\leq& 2 \left\| \tilde{\mathcal{L}}_{\phi_{\ast}}^n\left(\frac{\varphi + 2}{\int \varphi + 2 \ d\nu}\right) -h \right\| + 2\left\| \tilde{\mathcal{L}}_{\phi_{\ast}}^n( \mathbf{1}) - h \right\| \\\\
&\leq& 2 L\tau^n + 2L\tau^n = 4L\tau^n .
\end{eqnarray*}
To complete the proof it is enough to observe that the spectrum $\Sigma$ of  $\mathcal{L}_{\phi_{\ast}}$  is given by $\lambda \tilde{\Sigma}$ where $\tilde{\Sigma}$ is the spectrum of $\tilde{\mathcal{L}}_{\phi_{\ast}}$. 
\end{proof}

\section{Statistical behavior for the equilibrium of the projection map}\label{decay for G}

In this section we will prove Theorem~\ref{decaimento G} and Theorem~\ref{TCL G}. The exponential convergence of the transfer operator to the invariant density in the space of H\"older continuous observables will allow us to establish an exponential decay of correlations for the equilibrium state of $(G, \phi_{\ast})$.

\begin{Teo}\label{decaimento}
The equilibrium state $\mu_{\ast}$ associated to the system $(G, \phi_{\ast})$ has exponential decay of correlations for H\"older continuous observables: there exists $0< \tau < 1$ such that for all $\varphi \in L^1(\mu_{\ast})$ and $\psi \in C^{\alpha}(\mathcal{Q}) $ there exists a positive constant $K(\varphi, \psi) $ satisfying:
       $$\left|\int \left(\varphi \circ G^n\right) \psi \ d\mu_{\ast} - \int \varphi \ d\mu_{\ast} \int \psi\ d\mu_{\ast}  \right| \leq K(\varphi, \psi) \tau^n \quad \mbox{for all} \ n\geq 1.  $$
\end{Teo}
\begin{proof}
Recall that $h$, the eigenfunction of the transfer operator associated to the spectral radius $\lambda$, is bounded away from zero and infinity. Let us consider first the case $\psi \cdot h \in \mathcal{C}_{k, \delta}$ for $k$ large enough. Without loss of generality, suppose  
$\int \psi \  d\mu_{\ast} = 1$. Thus we have
\begin{eqnarray*}
\left| \int \left(\varphi \circ G^n\right) \psi \ d\mu_{\ast} - \int \varphi \ d\mu_{\ast} \int \psi\ d\mu_{\ast}\right| &\!\!=\!\!& \left| \int\! \varphi \cdot \lambda^{-n}\mathcal{L}_{\phi_{\ast}}^n\left( \psi \cdot h \right) \ d\nu - \int\! \varphi \ d\mu_{\ast}   \right| \\
&\!\!=\!\!& \int\! \varphi \cdot \left[\frac{\lambda^{-n}\mathcal{L}_{\phi_{\ast}}^n\left( \psi \cdot h \right)}{h} - 1 \right] \ d\mu_{\ast} \\
&\!\!\leq\!\!& \int\! \left| \varphi \right| \ d\mu_{\ast} \cdot \left\| \frac{\lambda^{-n}\mathcal{L}_{\phi_{\ast}}^n\left( \psi \cdot h \right)}{h} - 1 \right\|_0.  
\end{eqnarray*}

By Proposition~\ref{cota norma} there exists a positive constant $L$ such that 
$$\left\| \frac{\lambda^{-n}\mathcal{L}_{\phi_{\ast}}^n\left( \psi \cdot h \right)}{h} - 1 \right\|_0 \leq \left\|h \right\|_0 \left\| \lambda^{-n}\mathcal{L}_{\phi_{\ast}}^n\left( \psi \cdot h \right) - h \right\|_0 \leq  \left\|h \right\|_0  L \tau^ n. $$

In the general case fix $B=k^{-1}|\psi \cdot h |_{\alpha, \delta}$ and write 
$\psi\cdot h= \xi$ where
$$\xi=\xi_{B}^{+}-\xi_{B}^{-}\,\,\,\, \mbox{\and}\,\,\,\, \xi_{B}^{\pm}=\frac{1}{2}\left(|\xi|\pm\xi\right)+B.$$

Hence $\xi_{B}^{\pm}\in \mathcal{C}_{k,\delta}$ and we can apply the previous estimates to $\xi_{B}^{\pm}.$ By linearity the proposition holds.
\end{proof}

We point out that since the transfer operator converges to the density in the space of H\"older continuous observables, we can estimate the constant $K=K(\varphi, \psi)$ obtained in the last proposition as follows
\begin{equation} \label{constante K}
K(\varphi, \psi)\leq \tilde{K}\|\varphi\|_{1}\left(\|\psi\|_{1}+|\psi|_{\alpha, \delta}\right) = K(\psi) \|\varphi\|_{1}
\end{equation}
where  the constant $\tilde{K}$ does not depend on  $\varphi$ or on $\psi$ and $K(\psi)$ is a constant that depends only on $\psi$. 

Let $\mathcal{B}$  be the Borel $\sigma$-algebra of $\mathcal{Q}$ and denote $\mathcal{B}_n:= G^{-n}(\mathcal{B})$ for $n \geq 0$. A real function $\psi: \mathcal{Q} \to \mathbb{R}$ is $\mathcal{B}_n$-measurable if and only if there exists a $\mathcal{B}$-measurable function $\psi_n$ satisfying $\psi= \psi_n\circ G^{n}$. Moreover, we have the decreasing inclusion: $\mathcal{B}=\mathcal{B}_0 \supset \mathcal{B}_1\supset \cdots \supset \mathcal{B}_n\supset \cdots$. 
Let $\mathcal{B}_{\infty}$ be the intersection 
      $$  \mathcal{B}_{\infty} = \displaystyle\bigcap_{n \geq 0}\mathcal{B}_n. $$
An invariant probability measure $\mu$ is said to be \emph{exact} if every $\mathcal{B}_{\infty}$-measurable function is constant $\mu$- almost everywhere.

As a first consequence of the exponential decay of correlations we obtain the exactness property of the equilibrium measure associated to the system $(G, \phi_{\ast})$.

\begin{Cor} \label{exactness G}
The equilibrium state $\mu_{\ast}$ is exact. 
\end{Cor}

\begin{proof}
Given $\varphi \in L^1(\mu_{\ast})$ a $\mathcal{B}_{\infty}$-measurable function, for each $n \geq 0 $ there exists a $\mathcal{B}_n$-measurable function $\varphi_n$ such that $\varphi= \varphi_n \circ G^n$. In particular, $\left\|  \varphi_n \right\|_1 = \left\| \varphi \right\|_1$. By the decay of correlations, Theorem~\ref{decaimento G}, combined with \eqref{constante K} for any H\"older continuous function $\psi$ there exists $K(\psi) > 0$ such that 
\begin{eqnarray*}
\left|\int \left(\varphi - \int \varphi \ d\mu_{\ast} \right)\psi \ d\mu_{\ast}   \right|& =&  \left|  \int\left( \varphi_n \circ G^n \right)\psi \ d\mu_{\ast} - \int \psi \ d\mu_{\ast} \int \varphi \ d\mu_{\ast}  \right|\\\\
&\leq& K(\psi) \left\|\varphi  \right\|_1 \tau^n .
\end{eqnarray*}
Since the last term converges to zero when $n$ goes to infinity we have 
$$\int \left(\varphi - \int \varphi \ d\mu_{\ast} \right)\psi \ d\mu_{\ast} =0. $$
Since $\psi$ is arbitrary it follows that $\varphi = \int \varphi \ d\mu_{\ast} $ is constant $\mu_{\ast}$-almost everywhere.
\end{proof}

Notice that, in particular, the exacteness of $\mu_{\ast}$ implies its ergodicity.

In order to establish a central limit theorem for the equilibrium state of $(G, \phi_{\ast})$ we first state a non-invertible case of an abstract central limit theorem due to Gordin. For its proof one can see e.g. [\cite{Viana}, Theorem 2.11].

\begin{Teo}[Gordin]\label{TCL}
Let $\left( M, \mathcal{F}, \mu \right)$ be a probability space and $f: M\to M$ be a measurable map such that $\mu$ is an invariant ergodic probability measure. Let $\varphi \in L^{2}(\mu)$ such that $\int \varphi \ d\mu =0$. Denote by $\mathcal{B}_n$ the non increasing sequence of $\sigma$-algebras $\mathcal{B}_n = f^{-n}(\mathcal{B})$ and assume 
              $$ \displaystyle\sum_{n=0}^{\infty} \left\|\mathbb{E}(\varphi | \mathcal{B}_n)   \right\|_2 < \infty.$$
Then $\sigma \geq 0$ given by 
    $$\sigma^{2}= \int \varphi^2  \ d\mu + 2\displaystyle\sum_{n=1}^{\infty}  \int \varphi (\varphi \circ f^n) \ d\mu $$
is finite and $\sigma = 0$ if and only if $\varphi = u \circ f - u $ for some $u \in L^{2}(\mu)$. On the other hand, if $\sigma >0 $ then given any interval $A\subset \mathbb{R}$,
 $$\mu \Big(  x\in M : \frac{1}{\sqrt{n}} \displaystyle\sum_{j=0}^{n -1} \varphi(f^j(x))  \in A \Big) \to \frac{1}{\sigma \sqrt{2\pi}} \int_{A} e^{-\frac{t^2}{2\sigma^2}} \ dt$$
as $n$ goes to infinity.  		
\end{Teo} 

Now we derive from this result Theorem~\ref{TCL G}.
For each $n\geq 0$ denote by $L^2(\mathcal{B}_n)$ the set $L^2(\mathcal{B}_n)= \left\{  \psi \in L^2(\mu_{\ast}) : \psi \ \mbox{is} \ \mathcal{B}_n-\mbox{measurable} \right\}$. We have a sequence of inclusions $L^2(\mu_{\ast})=L^2(\mathcal{B}_0) \supset L^2(\mathcal{B}_1)\supset \cdots \supset L^2(\mathcal{B}_n) \supset \cdots$.

Since $L^2(\mu_{\ast})$ is a Hilbert space, given $\varphi \in L^2(\mu_{\ast})$ denote by $\mathbb{E}(\varphi|\mathcal{B}_n)$ the orthogonal projection of $\varphi$ to $L^2(\mathcal{B}_n)$. 
Let $\varphi$ be a H\"older continuous function such that $\int \varphi \ d\mu_{\ast} = 0$, then for all $n \geq 0$ we have
\begin{eqnarray*}
\left\| \mathbb{E} (\varphi|\mathcal{B}_n) \right\|_2 & = &\sup \left\{ \int\psi \varphi \ d\mu_{\ast} : \psi \in L^2(\mathcal{B}_n), \left\|\psi \right\|_2 = 1   \right\} \\
&=& \sup \left\{\int (\psi_n \circ G^n )\varphi \ d\mu_{\ast} : \psi_n \in L^2(\mu_{\ast}), \left\|\psi_n \right\|_2 = 1   \right\} \\
&\leq&K(\varphi)\left\|\psi_n \right\|_1 \tau^n  \leq  K(\varphi) \tau^n. %
\end{eqnarray*}

Note that in order to obtain the first inequality we apply the exponential decay of correlations from Theorem~\ref{decaimento G}. We warn the reader that when applying Theorem~\ref{decaimento G}, $\psi_n$ plays the role of $\varphi$ while $\varphi$ plays the role of $\psi$. To get the last inequality we used that $\left\|\psi_n \right\|_1 \leq \left\|\psi_n \right\|_2 = 1$. 

Therefore the series $\displaystyle\sum_{n\geq 0} \left\| \mathbb{E} (\varphi|\mathcal{B}_n) \right\|_2 $ is summable. 

Applying Theorem~\ref{TCL} we get a central limit theorem for the equilibrium state $\mu_{\ast}$ of $(G, \phi_{\ast})$. This proves Theorem~\ref{TCL G}.


\section{Statistical properties for equilibrium of horseshoes}\label{decay for F}

In the last section we have shown that the existence of a spectral gap for the transfer operator associated to the system $(G, \phi_{\ast})$ implies an exponential decay of correlations for the equilibrium $\mu_{\ast}$ on the space of H\"older continuous observables. Moreover, we also derived a central limit theorem for that equilibrium.

In this section we will use these results to derive similar statistical properties for the equilibrium state associated to the horseshoe $(F, \phi)$ defined in Section~\ref{Definitions}. We point out that since $F$ is a diffeomorphism we can consider its inverse $F^{-1}$ and from the way that we have defined the projection map $G$ we will state the results for $F^{-1}$. 

The key idea in this section is to disintegrate the equilibrium state for the horseshoe as a product of the equilibrium state for the system $(G, \phi_{\ast})$ and conditional measures on the stable fibers. For this we will use the following result due to Rohlin \cite{Ro}. The formulation stated here was given in \cite{Simons}.   

\begin{Teo}[Rohlin's Disintegration Theorem]\label{Rohlin}
Let X and Y be metric spaces, each of them endowed  with the  Borel $\sigma$-algebra. Let $\mu$ be a probability measure on $X$, let $\Pi : X \to Y $ be measurable and let $\hat{\mu}= \mu\circ \Pi^{-1}$. Then there exists a system of conditional measures $(\mu_y)_{y\in Y}$ of $\mu$ with respect to $(X;\Pi; Y )$, meaning that
\begin{enumerate}
\item $\mu_y$ is a probability measure on X supported on the fiber $\Pi^{-1}(y)$ for $\hat{\mu}$-almost every $y \in Y$.
 \item  the measures $\mu_y$ satisfy the law of total probability 
    $$ \mu(B) = \int \mu_y (B) \, d\hat{\mu}(y)  $$
for every Borel subset $B$ of $X$.    
\end{enumerate}
These measures are unique in the sense that if $(\nu_y)_{y\in Y}$ is any other system of conditional measures, then $\mu_y = \nu_y$ for $\hat{\mu}$-almost every $y \in Y $.
\end{Teo}

We point out that the conditional measures system in the last theorem is given by the weak$^\ast$ limit: 
     $$\mu_{y}= \lim_{\varepsilon\to 0} \mu_{\Pi^{-1}(B(y, \varepsilon))}, $$
where  $\mu_{\Pi^{-1}(B(y, \varepsilon))}$ is the conditional probability relative to ${\Pi^{-1}(B(y, \varepsilon))}$. Notice that $\mu_y$ is supported entirely on the fiber $\Pi^{-1}(y)$. 

In order to relate $F$ and $G$ we consider a projection of the parallelepipeds $\tilde{R}_0$ and $\tilde{R}_1$ onto the planes $P_0$ and $P_1$. See figure~\ref{planos2}.

\begin{figure}[!htb]
\centering
\includegraphics[scale=0.25]{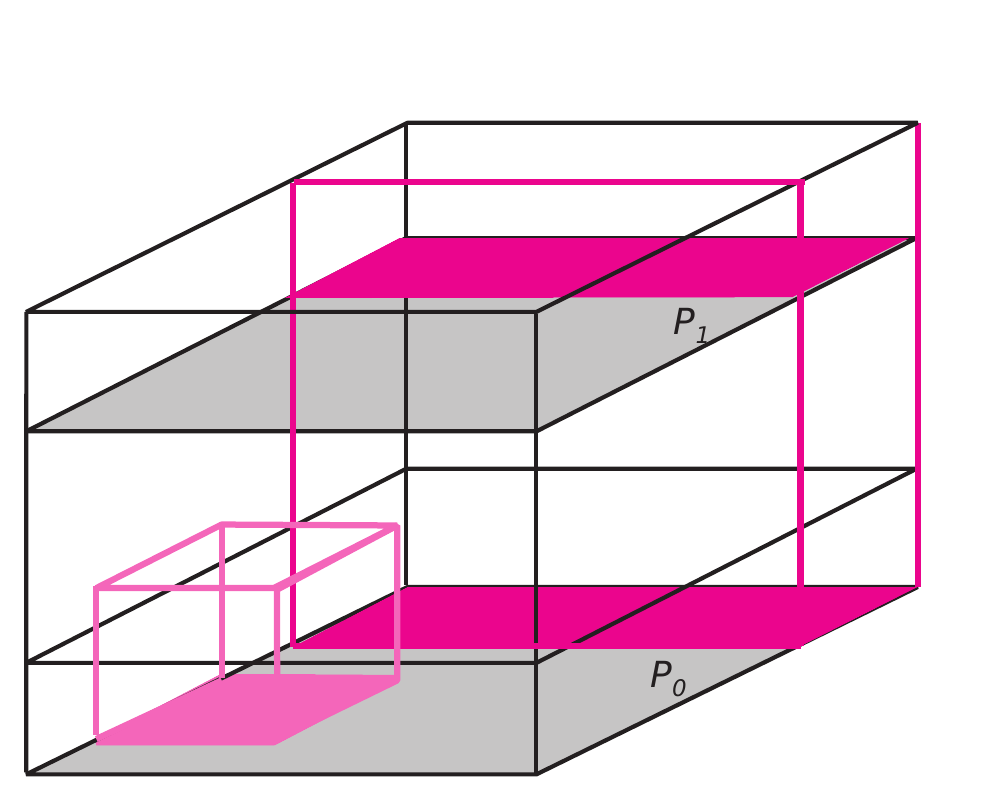}
\caption{Horizontal planes}\label{planos2}
\end{figure}

Define $\pi:\tilde{R}_1 \cup \tilde{R}_2  \to P_0 \cup P_1$ by 
$$\pi (x,y,z) := \left\{ \begin{array}{rc}
(x,y,0),  &\mbox{if}  \quad (x,y,z) \in \tilde{R}_0 \\
(x,y,\frac{5}{6}),  &\mbox{if}  \quad (x,y,z) \in \tilde{R}_1.
\end{array}
\right.
$$

Note that for each plane $z=z_0$ we have that $\pi^{-1}\circ \pi|_{z=z_0}= Id|_{z=z_0}$.
It is straightforward to check that $\pi$ is continuous, surjective and  $\pi \circ F^{-1}= G \circ \pi.$
Thus $\pi$ is a semiconjugacy between $F^{-1}$ and $G$.

Let $\phi: \tilde{R}_1 \cup \tilde{R}_2 \rightarrow\mathbb{R}$ be a Hölder continuous potential that does not depend on the $z$-coordinate, i.e.,  $\phi(x,y,\cdot):\tilde{R}_1 \cup \tilde{R}_2 \rightarrow\mathbb{R}$ is a constant function for every $x,y$ fixed. Hence $\phi$ induces a Hölder continuous potential $\phi_{\ast}:\mathcal{Q}\rightarrow\mathbb{R}$ defined on $\mathcal{Q}$ by \begin{equation}\label{potential induced}
\phi_{\ast}:=\phi\circ\pi^{-1}
\end{equation} 
which has the same variation as $\phi$.

In \cite{RS15} it was proved that when the variation of $\phi$ is smaller than $\frac{\log{\omega}}{2}$ there exists a unique equilibrium state $\mu_{\phi}$ associated to the horseshoe. Moreover, denoting by $\mu_{\ast}$ the equilibrium state of $(G, \phi_{\ast})$ where $\phi_{\ast}$ is given by \eqref{potential induced} this measure is the push-forward by $\pi$ of $\mu_{\phi}$. In other words for every Borel set $A$ of the $\sigma$-algebra on $\mathcal{Q}$ we have $$\mu_{\ast}(A)=\mu_{\phi}(\pi^{-1}(A)).$$

Recall that here the potential $\phi_{\ast}$ also satisfies condition~\eqref{condição potencial}. 

Consider $\Pi: \tilde{R}_1 \cup \tilde{R}_2  \to [0,1]$ the projection in the third coordinate $\Pi(x,y,z)=z$. 
Applying Rohlin's theorem we have for every Borel subset $B$ of $\tilde{R}_1 \cup \tilde{R}_2 $
   $$\mu_{\phi}(B)=\int \mu_z(B) \, d\hat{\mu}(z)$$
where $(\mu_{z})_{z\in [0,1]}$ is the system of conditional measures for the disintegration of  $\mu_{\phi}$ with respect to $(\tilde{R}_1 \cup \tilde{R}_2 , \Pi,[0,1])$.
In the next lemma we relate this system with the equilibrium state $\mu_{\ast}$ of the projection map.

\begin{Lema}\label{disintegration} Given $z\in [0, 1]$ and $A\subset\mbox{supp}(\mu_{z})$ a Borel subset of $\tilde{R}_1 \cup \tilde{R}_2 $ we have $$ \mu_{\ast}(\pi(A))=\mu_{z}(A).$$
\end{Lema}

\begin{proof} Fixing $z_{0}\in[0,1]$ and given $A\subset\mbox{supp}(\mu_{z_{0}})$ we have that $A$ is a subset of the plane $\{z=z_{0}\}$. Therefore
$$ \mu_{\ast}(\pi(A))=\mu_{\phi}(A)= \int \mu_{z}(A)\ d\mu(z)=\int \mu_{z_{0}}(A)\ d\mu(z)
=\mu_{z_{0}}(A).
$$
\end{proof}

Now we are able to prove the exponential decay of correlations for the equilibrium state $\mu_{\phi}$ associated to the horseshoe.

\begin{Teo}
The probability measure $\mu_{\phi}$ has exponential decay of correlations for H\"older continuous observables: there exists $0< \tau <1$ such that for every  $\varphi \in L^{1}(\mu_{\phi})$ and $\psi \in C^{\alpha}(\tilde{R}_1 \cup \tilde{R}_2 )$ there exists $K(\varphi, \psi)>0$ such that
$$\left| \int \left(  \varphi \circ F^{-n} \right)\psi d\mu_{\phi} - \int \varphi d\mu_{\phi} \int \psi d \mu_{\phi}\right| \leq K(\varphi, \psi)\tau^{n}  \quad \forall n\geq 1. $$
\end{Teo}

\begin{proof}
Let $\varphi \in L^1(\mu_{\phi})$ and $\psi\in C^{\alpha}(\tilde{R}_1 \cup \tilde{R}_2 )$ such that $\int \psi d\mu_{\phi}=1$. For each $n \in \mathbb{N}$ using Lemma~\ref{disintegration} we have
\begin{eqnarray*}
&&\left| \int\!\left(\varphi \circ F^{-n} \right)\psi d\mu_{\phi} - \int\! \varphi d\mu_{\phi}\int\!\psi d \mu_{\phi}\right| \\ \\ 
&=&\!\! \left| \int\! \left(\varphi \circ F^{-n} \right)\psi \ d\mu_{\phi} - \int\!\varphi d\mu_{\phi}\right| \\ \\
&\!\!=\!\!& \left| \int\!\!\int\! \left((\varphi \circ F^{-n} )\psi -\varphi \right) \, d\mu_z d \hat{\mu}(z)\right|\\\\
&\!\!=\!\!&\left|\int\!\! \int\!(\varphi\circ \pi^{-1}\circ \pi \circ F^{-n})(\psi \circ \pi^{-1}\circ \pi) - (\varphi\circ \pi^{-1}\circ \pi)  \,d\mu_z d \hat{\mu}(z)  \right| \\ \\
&\!\!=\!\!&\left|\int\!\! \int\!(\varphi\circ \pi^{-1}\circ G^{n} \circ \pi )(\psi \circ \pi^{-1}\circ \pi) - (\varphi \circ \pi^{-1}\circ \pi) \, d\mu_{z}  d\hat{\mu}(z)\right| \\ \\
&\!\!=\!\!&\left|\int\!\! \int\!\left((\varphi\circ \pi^{-1}\circ G^{n} )(\psi \circ \pi^{-1}) - (\varphi \circ \pi^{-1})\right)\circ \pi \, d\mu_{z}  d\hat{\mu}(z) \right| \\ \\
&\!\!=\!\!&\left|\int\!\! \int\!\left((\varphi\circ \pi^{-1}\circ G^{n} )(\psi \circ \pi^{-1}) - (\varphi \circ \pi^{-1})\right) \, d\mu_{\ast}  d\hat{\mu}(z) \right| \\ \\
&\!\!=\!\!&\left|\int\!\! \int\! \left( (\varphi_z\circ G^{n} )\psi_z - \varphi_z \right)  \, d\mu_{\ast} d\hat{\mu}(z) \right| 
\end{eqnarray*}

Note that for each $z$ fixed  $\psi_z= \psi \circ \pi^{-1}$ is a H\"older continuous function on $\mathcal{Q}$ and $\int \psi_{z}\ d\mu_{\ast}=\int \psi\ d\mu_{\phi}=1$. Also, $\varphi_{z}= \varphi\circ \pi^{-1}$ belongs to $L^1(\mu_{\ast})$. Then by the exponential decay of correlations property of $\mu_{\ast}$, Theorem~\ref{decaimento G}, there exists a positive constant $K(\varphi_z, \psi_z)$ and $0<\tau<1$ such that
\begin{eqnarray*}
\left|\int\!\! \int\! \left( (\varphi_z\circ G^{n} )\psi_z - \varphi_z \right) \, d\mu_{\ast} d\hat{\mu}(z) \right| &\leq& \int  K(\varphi_z,  \psi_z)\tau^{n} \, d\hat{\mu}(z)\\
&\leq& K(\varphi, \psi)\tau^{n}. 
\end{eqnarray*}
where $K(\varphi, \psi)$ is a uniform bound (in $z$) for $K(\varphi_z, \psi_z)$ obtained from equation~\eqref{constante K}.

We proved the desired inequality when $\psi\in C^{\alpha}(\tilde{R}_1 \cup \tilde{R}_2 )$ satisfies  $\int{\psi} d\mu_{\phi}=1$. For the general case it is enough to observe that
$$\left|\int\!(\varphi \circ F^n)\psi d\mu_{\phi} - \int\! \varphi d\mu_{\phi}\!\int\!\psi d \mu_{\phi}\right|
=\left|\int\!\psi d\mu_{\phi}\right|\!\left| \int\!(\varphi\circ F^n)\frac{\psi}{\int\!\psi d \mu_{\phi}} d\mu_{\phi} - \int\!\varphi d\mu_{\phi}\right|.$$
This ends the proof.
\end{proof}

Using equation~\eqref{constante K} and the $F$-invariance of  the equilibrium $\mu_\phi$ we can prove Theorem~\ref{decaimento F} from the last result:
\begin{eqnarray*}
&&\left| \int\!\left(\varphi \circ F^{n} \right)\psi d\mu_{\phi} - \int\! \varphi d\mu_{\phi}\int\!\psi d \mu_{\phi}\right| \\
&=&  \left| \int\!\left(\varphi \circ F^{2n}\circ F^{-n} \right)\psi d\mu_{\phi} - \int\! \varphi \circ F^{2n} d\mu_{\phi}\int\!\psi d \mu_{\phi}\right| \\ \\
&\leq &  K(\psi) \|\varphi \circ F^{2n}\|_{1} \, \tau^{n} = K(\psi) \|\varphi \|_{1} \, \tau^{n}.
\end{eqnarray*}

Since we have showed exponential decay of correlations for the equilibrium state $\mu_{\phi}$ it is straightforward to check that the same steps of the proof of exactness for the equilibrium associated to projection map $G$ (Corollary~\ref{exactness G}) hold in this context.

\begin{Cor} The equilibrium state $\mu_{\phi}$ is exact.
\end{Cor}

Finally consider $\mathcal{B}$ the Borel $\sigma$-algebra of $\tilde{R}_1 \cup \tilde{R}_2 $ and $\mathcal{B}_{n}$ the decreasing sequence defined by $\mathcal{B}_{n}=F^{-n}(\mathcal{B})$ for every $n\in\mathbb{N}.$ Let $\varphi$ be a H\"older continuous function satisfying $\int \varphi d\mu_{\phi}=0.$ Using the exponential decay of correlations of  $\mu_{\phi}$ we obtain that the series $\sum_{n\geq 0} \left\| \mathbb{E} (\varphi|\mathcal{B}_n) \right\|_2 $ is summable, where $\mathbb{E} (\varphi|\mathcal{B}_n)$ is the orthogonal projection of $\varphi$ to $L^{2}(\mathcal{B}_{n})$ for each $n\in\mathbb{N}$.

Applying Gordin's theorem we deduce that a central limit theorem holds for $\mu_{\phi}.$ Thus we have finished the proof of Theorem~\ref{TCL F}.


\subsection*{Acknowledgments} This work was carried out at Universidade do Porto. The authors are very thankful to Silvius Klein for the help with the manuscript version and encouragement. VR is grateful to Ivaldo Nunes for the encouragement.


\begin{thebibliography}{99}

%

\bibitem{Birkhoff} Birkhoff, G. {\em Lattice Theory}  American Mathematical Society Colloquium Publications 25 (4th ed.), Providence, R.I.: American Mathematical Society, (1979).

\bibitem{Baladi} Baladi, V., {\em Positive Transfer Operators and Decay of Correlations} World Scientific Publishing Co. Inc. (2000).

\bibitem{CN} Castro, A., Nascimento, T., {\em Statistical properties of the maximal entropy measure for partially hyperbolic attractors} Erg. Theory and Dyn. Sys. Available on CJO2016. doi:10.1017/etds.2015.86. 

\bibitem{CV} Castro, A., Varandas, P., {\em Equilibrium states for non-uniformly expanding maps: decay of correlations and strong stability}  Annalles de l'Institut Henri Poincar\'e - Analyse Non-Lin\'eaire, 225-249, (2013).

\bibitem{diazetal} D\'\i az, L., Horita, V., Rios, I., Sambarino, M., {\em Destroying horseshoes via heterodimensional cycles: generating bifurcations inside homoclinic classes}, Ergodic Theory and Dynamical Systems, 29 (2009) pp 433-474.


\bibitem{DG} D\"urr, D., Goldstein, S., {\em Remarks on the central limit theorem for weakly dependence random variables} Stochastic processes - mathematics and physics, Lect. Notes in Math., Vol 1158, Springer Verlag, Berlim, (1986) pp 104-118. 

\bibitem{FS} Ferrero, P., Schmitt, B., {\em Ruelle's Perron-Frobenius theorem and projective metrics} Coll. Math. Soc. J\'anos Bolyai, vol 27, (1979), pp 333-336. 

\bibitem{Ke} Keller, G. {\em Un theor\`eme de la limite centrale pour une classe de tranformations monotones par morceaux } C.R.A.S. A 291 (1980), pp 155-158

\bibitem{KN} Keller, G., Nowicki, T.,{\em Spectral theory, zeta functions and the distribution of periodic points for Collet-Eckmann maps} Comm. Math. Phys. (1992) 149, pp 31-69. 

\bibitem{Liverani} Liverani, C. {\em Decay of correlations} Ann. of Math. 142 (1995) pp 239-301. 

\bibitem{Liverani2} Liverani, C. {\em Decay of correlations for piecewise expanding maps} J. Stat. Phys. 78 (1995) 1111-1129.

\bibitem{RS15} Rios, I., Siqueira, J., {\em On equilibrium states for partially hyperbolic horseshoes}, Ergodic Theory and Dynamical Systems, 

\bibitem{Ro}  Rohlin, V. A., { \em On the fundamental ideas of measure theory}, Amer. Math. Soc. Translation,
(1952). no. 71, pp 55 .


\bibitem{Simons} Simmons, D., {\em Conditional measures and conditional expectation; Rohlin's disintegration theorem }, Discrete Contin. Dyn. Syst. 32 (2012), no. 7, 2565-2582.


\bibitem{Viana} Viana, M.,  {\em Stochastic dynamics of deterministic systems},  Col\'oquio Brasileiro de Matem\'atica, (1997).


\bibitem{Young} Young, L. S.,{\em Decay of Correlations for Certain Quadratic Maps}, Comm. Math. Phys. (1992) 146, pp 123-138 .  

\end{thebibliography}
\end{document}